\documentclass[noinfoline]{imsart}
\usepackage[utf8]{inputenc}
\usepackage{multicol, float}
\usepackage{hyperref, a4wide}
\usepackage{verbatim, euscript}
\usepackage{subfigure}
\usepackage{bbm,amssymb}
\usepackage{amsfonts}
\usepackage{amsmath, mathtools}
\usepackage{pictexwd}
\usepackage{amsthm}
\usepackage{graphicx}
\usepackage{ulem}
\usepackage{tikz}\usepackage{tikz}
\usetikzlibrary{arrows.meta}
 \usetikzlibrary{positioning}

         \newcommand*\myco
       {
       \raisebox{-0.3pt}
       {
       \begin{tikzpicture}       
         \draw (0,0) -- (0.25,0) -- (0,0.25)-- (0,0);   
          \end{tikzpicture}
       }
       }
\usepackage{rotating}
\usepackage{color}
\usepackage{hyperref}
\DeclareMathAlphabet{\mathpzc}{OT1}{pzc}{m}{it}
\definecolor{darkgrey}{rgb}{0.75,0.75,0.75}

\makeatletter
\renewcommand{\@fnsymbol}[1]{\@arabic{#1 }}
\makeatother

\makeatletter
\renewcommand{\@fnsymbol}[1]{\@arabic{#1 }}
\makeatother

\usetikzlibrary{arrows.meta}
\usepackage{rotating}
\usepackage{color}
\usepackage{hyperref}

\usepackage{mathtools}

\DeclareMathAlphabet{\mathpzc}{OT1}{pzc}{m}{it}

\theoremstyle{plain}
\newtheorem{thm}{Theorem}
\newtheorem{lem}{Lemma}[section]
  
\newtheorem{cor}[lem]{Corollary}
\newtheorem{defi}[lem]{Definition}
\newtheorem{bem}[lem]{Remark}

\newtheorem{prop}[lem]{Proposition}

\begin{document}

\begin{frontmatter}

 \title{Maintenance of diversity in a parasite population capable of persistence and reinfection}
 \runtitle{Maintenance of diversity}
  
  \begin{aug}
   
    \author{\fnms{Cornelia}
      \snm{Pokalyuk}\ead[label=e4]{pokalyuk@math.uni-frankfurt.de}}$\!\!$
    \, and \author{\fnms{Anton}
      \snm{Wakolbinger}\ead[label=e4]{wakolbin@math.uni-frankfurt.de}
      \ead[label=u4]{http://www.math.uni-frankfurt.de/$\sim$ismi/wakolbinger/}}
  
  \runauthor{Pokalyuk, Wakolbinger}

  \address{Institut für Mathematik\\
    Johann-Wolfgang Goethe-Universität\\
    60054 Frankfurt am Main \\
    Germany\\
   % \printead{e4}
    %\\
    %\printead{u4}
  }
\end{aug}

\begin{abstract}
Inspired by DNA data of the human cytomegalovirus we propose a model of a two-type parasite population distributed over its hosts.
The parasite is capable to persist in its host till the host dies,
and to reinfect other hosts. To maintain type diversity within a host, balancing selection is assumed.

For a suitable parameter regime we show that in the limit of large host and parasite populations the host state
frequencies follow a dynamical system with a globally stable equilibrium, guaranteeing that both types are maintained 
in the parasite population for a long time on the host time scale.
\end{abstract}


\end{frontmatter}

\section{Introduction}\label{Intro}

Diversity is essential for the survival of species, see e.g. \cite{Frankham2005}. This applies in particular to parasites.
An interesting example is the human cytomegalovirus (HCMV), an old herpesvirus, which is carried by a substantial
fraction of mankind  (see \cite{Cannon2010}) and in general leads
to an asymptomatic infection in the immunocompetent host (see \cite{Griffiths1984,Zanghellini1999}).
In DNA data of HCMV a high genetic diversity is observed in coding regions, see \cite{Goerzer2011}. This diversity can be helpful to resist the defense of the host.
Furthermore, for guaranteeing its long term survival, HCMV seems to have developed elaborate mechanisms
which allow it to
persist lifelong in its host and to establish reinfections in already infected hosts. We propose a model to study the effects
of these mechanisms on the maintenance of diversity
in a parasite population. A central issue hereby is that the diversity of the (surrounding) parasite population can be introduced
into single hosts. 

In our model we assume for the sake of simplicity that each infected host carries a constant number $N$ of parasites
until its death, and that only two types of parasites exist, type~$A$ and type~$B$. 
We consider only the population of infected hosts and assume that its size $M$ is constant. 
The evolution of the frequency of type $A$ in each of the $M$ hosts is driven by three factors: a)~parasite reproduction,
b) host replacement,
and c) reinfection. Within hosts, parasites reproduce subject to balancing selection with a drift 
to an equilibrium frequency of the two types.
Whenever a host dies, it is replaced by a new, so far uninfected host, which instantly suffers
a primary infection from a randomly chosen infected host. At such a primary infection the host is infected with a single type chosen
randomly according to the type frequencies in the infecting host; the type $A$-frequency  in the primary infected 
host is then instantly set to either $1$ or $0$. At reinfection a single parasite in the reinfected host is replaced by a randomly chosen parasite
transmitted from the infecting host.

This scenario can be interpreted in classical population genetics terms as a population distributed over $M$ islands and migration between islands.
Within each island reproduction is panmictic and driven by balancing selection or alternatively (in a diploid setting)
by a model of overdominance, i.e. heterozygote advantage, see \cite{Gillespie2004}.  

Host replacement events (which model the death 
of a host and its replacement by a primary infected host) can be 
interpreted as extinction-recolonization events; the role of such events on the reduction of neutral diversity was studied e.g.~in \cite{PannelletAl1999}.
 Since in our model a host after primary infection carries either only type $A$ or type $B$-parasites, host replacement leads to a reduction of polymorphic hosts, i.e.~of hosts that  simultaneously carry both types of parasites.
Furthermore, host replacement produces fluctuations in the host type frequency,  which eventually
leads to the extinction of one parasite type. 

The role of balancing selection in evolution is still a matter of debate. It has been proposed that host-parasite coevolution is one of the main forces driving immune genes
to evolve under balancing selection, see e.g.~\cite{CrozeEtAl2016}. The host defense system, e.g.~the major histocompatibility complex (MHC) of vertebrates, exhibits a large diversity 
and MHC genes show patterns of balancing selection, see \cite{Leffler2013}.

In our hierarchical model for the evolution of the parasite population 
we study the effect of balancing selection on the diversity in parasite populations and on the spread of this diversity in 
the host population; hereby  effects on the level of the host population
and on the level 
of the parasite population
are taken into account. Related hierarchical models have been studied from a mathematical perspective 
e.g.~by \cite{Dawson2018, Luo2017}. In these papers an emphasis on models for selection on two scales is made and phase transitions
(in the mean-field limit) are studied
at which particularly the higher level of selection (group selection) can drive the evolution of the population.
In our model, balancing selection is only acting on the lower level (within-host parasite populations), but we focus on parameter regimes in which balancing
selection is also lifted to the higher level,  
such that both parasite types are maintained in the host population
for a long time,
in which hosts carrying a single parasite type only, as well as hosts carrying both types 
of parasites are present in the population. (This  corresponds to a scenario  observed in samples of HCMV hosts, see e.g. \cite{NovakEtAl2008, ParadowskaEtAl2014, RossEtAl2011}). 

It turns out that for large parasite and large host populations this scenario applies if 

-  the effective reinfection rate, that is the rate at which in so far single-type infected hosts
a second type is established, acts on the same time scale as host replacement, 

- balancing selection is strong enough to keep the type frequencies within a host close to the equilibrium frequency,
once a host was effectively reinfected, and

- parasite reproduction is much faster than host replacement, and a mild upper bound on the parasite  reproduction rate is imposed.

Under corresponding assumptions on the model parameters  we show that  
on the (microscopic, within-host) parasite time scale balancing selection maintains diversity in the host population
also on the (macroscopic) host time scale. 
Within a typical host the evolution of type frequencies  can be separated into two alternating phases: 
1) A host infected with a single type
remains in this state until she is affected by a successful reinfection event or a host replacement event, and 2) a host carrying both types close to
the equilibrium frequency waits for a replacement event that substitutes her by a pure-type host again.
We will identify the limiting random genealogies of typical hosts by using graphical representations of the random genealogies of hosts in the 
model with a finite (but large) number of parasites. Furthermore, we obtain also a limit law for
the  dynamics of the states of the hosts as the number of hosts becomes large, and
identify the deterministic dynamical system that governs the host-state frequencies.

Essential quantities to show the concentration on the two pure frequencies and the equilibrium frequency  are the {\it probability of balance}, i.e.
the probability with which a reinfection event leads to 
the establishment of the second type in a so far single-type infected host, and the {\it time to balance}, i.e.~the
time needed to reach (a small neighborhood 
of) the equilibrium frequency $\eta$ after reinfection. These quantities determine the parameter regimes 
in which we can observe the described scenario. Similar to the the case of positive selection (see e.g.~\cite{Champagnat2006, Gonzalez2016, PokalyuketAl2013}), 
branching process approximations 
as well as approximations by deterministic ODE's 
can be used to estimate these probabilities and times. 
A notable difference compared to the situation
studied e.g. in  \cite{Gonzalez2016} is a modification in Haldane's formula: the {\it probability to balance}
(when starting from frequency $0$) 
differs from the fixation probability in the case of positive selection by the factor $\eta$ (which is the equilibrium frequency), see Lemma \ref{balprob}.
Furthermore the time to balance is
longer than the time of a selective sweep in the corresponding setting.
This is due to the fact that random fluctuations close to the equilibrium are larger than fluctuations
close to the boundary.

\section{Model and Results}\label{modres}

\subsection{Model}
Let $M,N\in \mathbb N$. We model the evolution of the parasite
population distributed over $M$ hosts by a $\{0,\tfrac1N,\ldots,1 \}^M$-valued Markovian jump process
${\bf X}^{N,M}= (X^{N,M}_1(t), ..., X^{N,M}_M(t))_{t \geq 0}$, where $X^{N,M}_i(t)$, $1\le i\le M$,  represents the relative frequency
of type $A$-parasites in host $i$ at time $t$.  (As long as there is no ambiguity, we suppress the superscripts $N,M$ in $X^{N,M}_i(t)$ and ${\bf X}^{N,M}$.)
Before stating the jump rates in \eqref{ViralModel} we describe the dynamics of ${\bf X}^{N,M}$ in words. The host population as well as the parasite population within each host follow dynamics which both are modifications of the classical
Moran dynamics, see \cite{Ewens2004}, Chapter 3.4. he
reproduction rate of parasites is assumed to be $g_N$ times larger 
than the rate of host replacement. The parasite population within a host experiences 
 balancing selection towards an equilibrium frequency $\eta$, for some fixed $\eta \in (0,1)$.  More specifically, in host $i$ parasites of type $A$, when having relative frequency $x_i$, reproduce at rate
$g_N(1+s_N(\eta - x_i))$ and those of type $B$ at rate $g_N(1+s_N(\eta-x_i))$, where $s_N$ is a small positive number. Thus the rate of reproduction of type $A$-parasites
is larger (smaller) than that of
type $B$-parasites,
if the frequency of type $A$ is below (above) the equilibrium frequency $\eta$, at which type $A$ and type $B$ are balanced.
At a reproduction event a parasite splits into two and replaces a randomly chosen parasite from the same host. Thus a change in 
frequency  occurs only if the type of the reproducing parasite differs from the type of the parasite which is replaced.
Reinfection events occur at rate $r_N$ per host; then a single parasite in the reinfecting host (both of which are randomly chosen) is copied and transmitted to the reinfected
host. At the same time  a randomly chosen parasite is instantly removed from this host; in this way the parasite population size in each of the hosts is kept constant.
A reinfection
event leads to a change in frequency in the reinfected host only if the type of the replaced parasite differs from the transmitted one. Hence, 
if ${\bf X}^{N,M}= (x_1,..., x_M)$, then  the frequency of type $A$ in host $i$ changes due to reinfection at rate 
$r_N   \frac{1}{M} \sum_{j=1}^M x_j (1-x_i)$ by an amount of $+1/N$, and at rate $r_N \frac{1}{M} \sum_{j =1}^M (1-x_j) x_i$ by an amount $- 1/N$. 
If an infected host dies it is replaced by a so far uninfected host, which instantly is infected by a randomly chosen infected host. 
Then  only a single type is transmitted,
leading for each host to jumps of the type $A$-frequency to 1 and 0
at rate   $\bar {\bf x}$ 
and $1- \bar{ \mathbf x}$, respectively, with  
$\bar{\bf x}:=  \frac{1}{M}\sum_{j=1}^{M} x_j$.

To summarize, jumps from state 
$\textbf{x} = (x_1, ..., x_M)  \in \{0,\tfrac1N,\ldots,1 \}^M$ occur for $i=1, ..., M$ 
\begin{align}
&\mbox{to }  \textbf{x} + \frac{1}{N} e_i  &  \mbox { at rate }  &  g_N  (1+  s_N (\eta- x_i)) N x_i (1- x_i) +  r_N   \frac{1}{M} \sum_{j=1}^M x_j (1-x_i)  \notag
  \\  
&\mbox{to } \textbf{x} - \frac{1}{N} e_i  &  \mbox { at rate }  & g_N  (1+ s_N ( x_i-\eta)) N x_i (1- x_i) + r_N \frac{1}{M} \sum_{j =1}^M (1-x_j) x_i   \label{ViralModel} \\
&\mbox{to } \textbf{x} + (1- x_i)e_i &  \mbox { at rate }   & \bar{ {\bf x}} , \qquad \mbox{ and }
\mbox{to } \textbf{x} - x_i e_i \qquad   \mbox { at rate }    (1- \bar{ \bf x} ), \notag
\end{align}
with  
$e_i=(0, ..., 1, ..., 0)$
the $i$-th unit vector of length $M$.

\begin{bem}
The biological relevance of this model is discussed in detail in
the companion paper \cite{CPBio}. Briefly summarized: When analyzing DNA samples  of the human cytomegalovirus it is striking that many coding
regions cluster into a few, phylogenetically distant haplotypes, see \cite{Goerzer2011} and the references therein.
Given that these haplotypes lie in coding regions, the underlying fitness landscape could be sharply peaked. Under this assumption it is likely that
genetic drift erases haplotype diversity before it is repaired by mutation. In light of the contrary observation, one might suppose that haplotypes
are actively maintained in the viral population, as we do in our model.

The major motivation
for the above described model thus comes from observations of DNA samples of HCMV. However, as simultaneous infections by
several genotypes or even by several
species appear to be the rule rather than the exception, see \cite{Petney1998, Lord1999}, the scenario discussed here might be relevant (suitably
generalized)
also for other host-parasite systems.
\end{bem}

In the following we will specify the assumptions on the strength of selection and intensity of reinfection and parasite reproduction relative to host replacement.
For strong enough selection we show
that, in the limit of an infinitely large parasite population per host, only three 
states of typical hosts exist, those infected with only one of the types $A$ or $B$ and those infected with both types, where  $A$ 
is at frequency $\eta.$ These three host states will be called the {\it pure states} (if the frequency of type $A$ in a host is 0 or 1)
and the {\it mixed state} (if the frequency of type $A$ in a host is $\eta$).

Only reinfection events can change a host state from a pure to the mixed state. In most cases 
reinfection events are not effective, in the sense that these events cause only a short excursion from the boundary frequencies 0 and 1.
We will see that if the selection strength and reinfection rate are appropriately scaled, the effective reinfection rate acts on the same time
scale as host replacement. Furthermore, if selection is of moderate strength and parasite reproduction is fast enough (but not too fast), 
transitions of the boundary frequencies to the equilibrium frequency~$\eta$ will appear as jumps on the host time scale and transitions between
the host states 0, $\eta$ and 1 are only caused by host replacement and effective reinfection events.
See also Figure \ref{typhost} for an illustration of the type $A$-frequency path in a typical host for large $N$.

We will see in Theorem \ref{diversity} that if the effective reinfection rate is larger than a certain bound depending on $\eta$ then
there exists in the limit $N\rightarrow \infty$ and $M\rightarrow \infty$ a stable equilibrium of the relative frequencies of hosts of type 0, $\eta$ and 1, 
at which both types of parasites are present in the entire parasite population at a non-trivial frequency.

The precise assumptions on the parameters are as follows:
\bigskip

\textbf{Assumptions ($\mathcal{A}$)}: There exist  $b\in (0,1)$, $r>0$ and $\epsilon >0$ such that the parameters $s_N$, $r_N$ and $g_N$ in \eqref{ViralModel} obey
 \begin{itemize}
 \item[$(\mathcal{ A} 1)$] (moderate selection) \[s_N = N^{-b},\] 
 \item[$(\mathcal{ A} 2)$] (frequent reinfection) 
\[\lim_{N\rightarrow \infty} r_N s_N = r,\] 
 \item[$(\mathcal{ A} 3)$] (fast parasite reproduction) 
 \begin{itemize}
 \item[i)] (bound from below) \[   1/g_N = o(N^{-3b -\epsilon}),\]
 \item[ii)] (bound from above)  \[ g_N = \mathcal O(\exp(N^{1 - b (1+\epsilon)})).\]
  \end{itemize}
  \end{itemize}

\begin{bem}\label{Rem2_2}
\begin{itemize}
 \item [i)]  Assumption $(\mathcal{ A} 1)$  implies that 
 \[\lim_{N\rightarrow \infty} s_N =0\] and 
 \[\lim_{N\rightarrow \infty} s_N N = \infty.\]
\item [ii)]   Assumption $(\mathcal{ A} 1)$, $(\mathcal{ A} 2)$ $(\mathcal{ A} 3)$ together imply that for large $N$
\[1 \ll r_N \ll g_N, \]
 This says that hosts experience frequent reinfections during their lifetime and between two reinfection events many parasite reproduction events happen.  Such a parameter regime seems realistic;
 see \cite{CPBio} for additional discussion.
 \item  [iii)]{\color{black} As we will see in Section \ref{proofs}, Assumption $(\mathcal{ A} 3)$ implies that there exists a sequence  of neighbourhoods $U^{\eta,N} \downarrow \{\eta\}$   
 and a sequence $\delta_N \downarrow 0$ as $N\rightarrow \infty$, such that typical type frequencies within hosts are asymptotically concentrated on the sets $[0, \delta_N) \cup U^{\eta,N} \cup (1-\delta_N, 1]$.  }
 \item [iv)] {\color{black} Assumption ($\mathcal{ A} 3$\rm{i}) implies that parasite reproduction is fast enough that on the host time scale a transition from the boundary to the equilibrium frequency $\eta$ (if it occurs) is instantaneous in the limit $N\rightarrow \infty$. Assumption ($\mathcal{ A} 3$\rm{ii})  implies that for large $N$, with high probability, only host replacement (but not random fluctuation caused by parasite reproduction) can bring the frequency of type $A$ from close to the equilibrium frequency $\eta$ to the boundary states $0$ and $1$.}
  \end{itemize}
\end{bem}
\begin{bem}\label{rem2_3}
 Remark \ref{Rem2_2}(iii) does not yet make a statement on how quickly the sequence $U^{\eta,N}$ shrinks to $\eta$ as $N\to \infty$.  From the point of view of applications, see \cite{CPBio},  one is interested also in the  size (or at least the order of magnitude)  of the $U^{\eta,N}$ for large $N$.
In fact, the proofs  in Section \ref{proofs}  work with the  following choice of $U^{\eta,N}$:
\begin{align}\label{Ueta}
U_a^{\eta,N}:= (\eta -s^a_N, \eta + s_N^a),   \quad a > 0.
\end{align} This will require the following strengthening of conditions $(\mathcal{A})$: 
In addition to  $(\mathcal{ A} 1)$ and $(\mathcal{ A} 2)$ we have for some 
 $a \in (0, \frac{1-b}{2b})$
 \\
 $(\mathcal{A} 3')$

{\rm i)}
 \[1/g_N = o(N^{- b(3 \vee (2+a)) -\epsilon})\]
\indent and 

 {\rm ii)}
\[g_N = O(\exp(N^{1- b(2a +1 +\epsilon)})).\]

Note that $(\mathcal A 3)$ always implies the existence of some sufficiently small constant $a >0$ such that $(\mathcal{A} 3')$ is satisfied. On the other hand, the larger the constant $a$ in Assumption $(\mathcal{A} 3')$ is (and the more restrictive asymptotic bounds on $g_N$ one therefore has  compared to those in $(\mathcal{A} 3)$), the smaller the  $U^{\eta,N}_a$ in \eqref{Ueta} will be.
\end{bem}

In the following will analyze the cases ``$N\to \infty$ with $M$ fixed" (Theorem \ref{TfiniteM}), ``first $N\to\infty$, then $M\to\infty$" 
(Corollary \ref{chaos1}) and "$N\to\infty$, $M\to\infty$ jointly'' (Theorem \ref{MF}).

\subsection{Large parasite population, finite host population}\label{Sec2.2} 

Let $M\in \mathbbm{N}$. We prepare our first main result
by defining the $\{0,1,\eta\}^M$-valued Markovian jump process
 $\textbf{Y}^M=(Y^1_t, ..., Y^M_t)_{t\geq 0}$, which will turn out to be the process of type $A$-frequencies
 in hosts $1, \dots, M$ in the limit $N\rightarrow \infty$. From the state 
 $\mathbf{y}= (y_1, ...., y_M)$, the process $\textbf{Y}^M$ jumps by flipping for $i\in \{1,\ldots,M\}$ the component $y_i$ 
 \begin{alignat}{3}
& \text{\quad from \quad } 0 \text{ or } \eta \text{\quad to \quad }1 && \text{\quad at rate \quad }   \frac{1}{M} \sum_{j=1}^M y_j, \nonumber \\ 
&\text{\quad from \quad } 1 \text{ or } \eta \text{\quad to \quad }0 && \text{\quad at rate \quad }  \frac{1}{M} \sum_{j=1}^M (1- y_j), \label{Yrates} \\
& \text{\quad from \quad } 0 \qquad \text{\quad to \quad }\eta && \text{\quad at rate \quad }\frac{ 2r \eta }{M} \sum_{j =1}^M y_j, \nonumber \\
& \text{\quad from \quad } 1 \qquad \text{\quad to \quad }\eta && \text{\quad at rate \quad } \frac{ 2r (1-\eta)}{M} \sum_{j =1}^M (1-y_j).\nonumber
\end{alignat}

\begin{bem}[Graphical representation of $\textbf{Y}^M$]\label{graphYM}
  The process $\textbf{Y}^M$ has a graphical representation which explains the jump rates in terms of the underlying hierarchical structure and will
  also be instrumental in the proof of Theorem \ref{TfiniteM}. {\color{black} See Figure \ref{graphY} for an illustration.} This representation has two main ingredients:

  {\color{black}
1) the host replacement ({\it HR}) events: for each pair $(i, j) \in \{1,\ldots, M\}^2$  there is a 
Poisson process of rate $1/M$ on the time axis. At any  time point  $t$ of this Poisson process, if host $j$ is   at time $t-$
in state $0$ or $1$, then host $i$ adopts that state at time $t$; if, however, host $j$ is  at time $t-$ in state $\eta$, then that state is set  to 1  with probability $\eta$, and set to $0$ with probability $1-\eta$. 

2) the potential effective reinfection ({\it PER}) events: for each pair $(i, j) \in \{1,\ldots, M\}^2$ 
there is a Poisson process of rate $2r/M$ on the time axis. At any  time point  $t$ of this Poisson process, \\ a) if at time $t-$ host $j$ is in state $1$ and host $i$ is in state $0$, then at time $t$ host $i$  with probability~$\eta$ jumps  to state $\eta$, and with probability $1-\eta$
remains in state $0$, \\ b) if at time $t-$  host $j$ is in state $0$ and host $i$ is in state $1$, then at time $t$ host $i$ jumps  to state
$\eta$ with probability $1-\eta$, and with probability $\eta$ remains in state $1$, \\ c) if at time $t-$  host $j$ is in state $\eta$ and  host $i$ is in state $0$, then  with probability $\eta^2$
host $i$ jumps to state $\eta$, and with  probability $1-\eta^2$ stays in state $0$,  \\ d) 
if at time $t-$  host $j$ in state  $\eta$ and  host $i$ is in state $1$, then with probability $(1-\eta)^2$
host $i$ jumps to state $\eta$  and  with  probability $1-(1-\eta)^2$ stays in state $1$, and \\
e) in the remaining cases nothing changes. 

With the above described rules, we may think of a sequence of independent coin tosses (with success probability $\eta$) attached to  the HR and the PER events; note that, corresponding to the rule described in 2c) and 2d),  PER events in which host $j$ is in state $\eta$, and host $i$ is either in state $0$ or in state $1$, require two independent coin tosses, each with success probability $\eta$. 
The host state configuration $((Y_i^M(t))$ is then determined from the
the initial host state configuration $((Y_i^M(0))$ together with the realizations of the Poisson processes and of the coin tosses.} 
\end{bem}

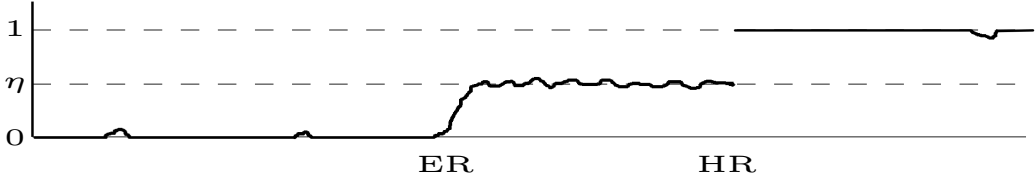
\begin{figure}
 \resizebox{15cm}{2.5cm}{
 \begin{tikzpicture}
 \node at (0,-.07) {\includegraphics[width=6cm, height=1cm]{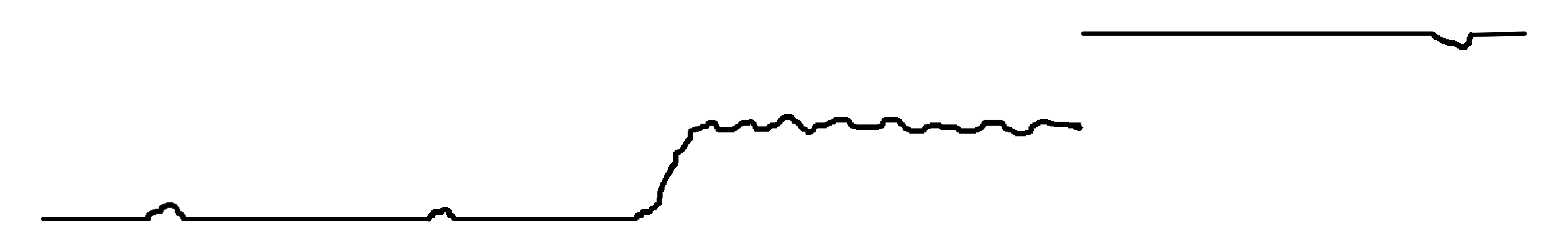}};
 \draw (-2.85,-0.5)-- (-2.85,0.5);
 \draw[line width=0.001 cm] (-2.85, -.5)--(2.8,-.5);
  \draw[line width=0.001 cm, dashed] (-2.85, -.11)--(2.8,-.11);
   \draw[line width=0.001 cm, dashed] (-2.85, .29)--(2.8,.29);
 \node at (-2.95, -0.5) {\tiny 0 };
 \node at (-2.95, -0.12){\tiny$\eta$};
 \node at (-2.95, 0.299){\tiny 1};
  \node at (-0.5, -0.7){\tiny ER};
   \node at (1.1, -0.7){\tiny HR};
 \end{tikzpicture}
}
\caption{Frequency path in a typical host for finite but large $N$. Ineffective reinfections cause small excursions of the frequency path from the boundary. An effective reinfection (ER)  is followed by a quick transition (consisting of small jumps of the frequency path) from the boundary to a neighbourhood of the equilibrium frequency $\eta$. When the frequency path is close to the equilibrium, parasite reproduction causes random fluctuations around the equilibrium frequency $\eta.$ Host replacements (HR) cause jumps to the frequencies 1 and 0.}
\label{typhost}
\end{figure}

 \begin{figure}
 \begin{center}
 \resizebox{10cm}{5cm}{
\begin{tikzpicture}
\draw (-2,0) -- (-2,5);
\node at (-2.2, 0.1) {0};
\node at (-2.2, 4.9) {$t$};
\node at (-2, 5.3) {\small time};

 \draw (0,0) -- (0,5); 
 \draw (2,0) -- (2,5);
 \draw (4,0) -- (4,5);
 \draw (6,0) -- (6,5);
 \draw (8,0) -- (8,5);
% \draw [->] (4,1)-- (2,1);
   \draw[arrows= -{To[scale=2]}] (4,1) --(2,1);
  \node at (4.2,1) {$A $};
 %\draw[ ->]  (0,3) -- (2, 3);
  \draw[arrows= -{To[scale=2]}] (0,3) --(2,3);

 \node at (3.7,2) {$A$,1 };
%  \draw [dashed, ->] (8,4) -- (6,4) ;
    \draw[arrows= -{To[scale=2]}, dashed] (8,4) --(6,4);
 % \draw[ <-]  (6,3) -- (8, 3);
% \draw [dashed, ->] (4,2) -- (6,2) ;
 \draw[arrows= -{To[scale=2]}, dashed] (4,2) --(6,2);

 \node at (0,-.2) {0};
 \node at (2,-.2) {1};
  \node at (4,-.2) {$\eta$};
   \node at (6,-.2) {0};
    \node at (8,-.2) {$\eta$};

  \node at (0,5.2) {0};
 \node at (2,5.2) {0};
  \node at (4,5.2) {$\eta$};
   \node at (6,5.2) {$\eta$};
    \node at (8,5.2) {$\eta$};
 
 %\node at 
 
 \end{tikzpicture}
 }
 \end{center}
\caption{An illustration of the graphical representation of $Y^M$. Solid arrows indicate 
host replacement events and dashed arrows stand for potential effective reinfection events. 
If at a host replacement event the incoming line is of type $\eta$, then a coin toss decides which parasite type ($A$ or $B$) is transmitted: {\color{black} type $A$ is transmitted  with probability~$\eta$, and type $B$ is transmitted with probabilit $1-\eta$ .} The host that suffers
such a primary infection then instantly takes the state $1$ or $0$ {\color{black}(i.e. the frequency of type $A$ is 1 and 0, respectively)}. The outcome of the coin toss ($A$ or~$B$) is annotated by 
the letter next to the tail of the arrow.
If at a reinfection event the incoming line is of type $\eta$, then two coin tosses are necessary, 
one to decide which type is transmitted (the letter ($A$ or $B$) next to the tail of the arrow indicates the result of this coin toss) 
and one coin toss to decide if the transmitted type can establish itself in the infected host (the digit 1 stand for ''yes`` and the digit 0 for ''no``).
At time $0$ lines are randomly typed with 0, $\eta$ and 1 according to some initial distribution. At time $t$ 
the propagated types are displayed.
In this example, for  two of the arrows (the ones without letters/digits at their tails) no coin tosses are necessary to decide the state of the continuing line. }
\label{graphY}
\end{figure}
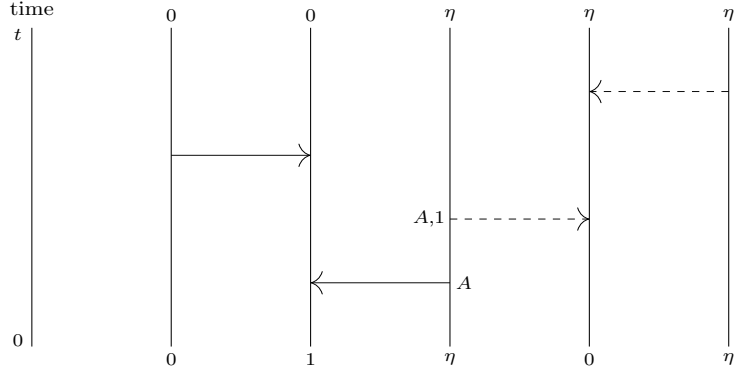

\begin{thm}\label{TfiniteM}
Let $\mathbf X^{N,M}$ be the $\{0,\frac 1N,\ldots,1\}^M$-valued process with jump rates \eqref{ViralModel}. Fix 
$M \in \mathbbm{N}$ and assume that the law of
 $\mathbf X^{N,M}(0)$ converges weakly  as $N\to \infty$ to a distribution $\rho$ concentrated on  $(\{0\}\cup [\alpha,1-\alpha] \cup \{1\})^M$ for some $\alpha>0$. Let $\textbf{Y}^M$  be the process with jump rates \eqref{Yrates}, and
 with the  distribution of $\textbf{Y}^M(0)$ being the image of~$\rho$  under the mapping 
 $0 \mapsto 0$, $1 \mapsto 1$, $[\alpha, 1-\alpha] \ni x \mapsto \eta$.
% Then
% $\textbf{X}^{N,M}$ converges in distribution to $\textbf{Y}^M$ on every finite time interval $[0,t]$ with respect to the Skorohod topology.
 Then under conditions ($\mathcal A$), for any $0 <\underline t <\overline t<\infty$,
the process $\textbf{X}^{N,M}$ converges as $N\to \infty$ on the time interval~$[\underline t,\overline t]$  in distribution with respect to the Skorokhod $\mathrm M_1$-topology to the process $\textbf{Y}^M$.
 \end{thm}
 
\begin{bem} The Skorokhod $\mathrm M_1$-topology , which is coarser than the more common \mbox{$\mathrm J_1$-topology},  adequately describes the mode of convergence of $\mathbf X^{N,M}$   to the jump process $\mathbf Y_M$ as $N\to \infty$.  For a definition and characterization of these topologies see \cite{Skorokhod1956}; see also  \cite{Collet2012} for a convergence theorem in the context of of adaptive dynamics which uses the \mbox{$\mathrm M_1$-topology}. 
 Theorem \ref{TfiniteM} reveals that in the limit $N\to \infty$ a Poissonian structure of jumps from the boundary points $0$ and $1$ to the equilibrium frequency $\eta$ emerges; these jumps capture the outcomes of effective reinfections. 
 In a graphical representation of $\mathbf X^{N,M}$ which is analogous to that of $\mathbf Y_M$ given in Remark \ref{graphYM} we will see how this Poissonian structure arises  from the limiting behavior of ``excursions from $0$ and from $1$'' (caused by reinfections) of the components $X^{N,M}_i$ together with the action of the host replacement, and that the components $Y^M_i$ are concentrated on $\{0,\eta,1\}$. 
Indeed, 
as soon as a component $X^{N,M}_i$ (i.e. the frequency of type $A$-parasites in host $i$) is appreciably away from $0$ and $1$,  then, for large $N$, the effect of the balancing selection, combined with the assumption of large $g_N$, is strong enough to take $X^{N,M}_i$ close to $\eta$ in a short time in a nearly monotonic way (and instantaneously to $\eta$ in the limit $N\to \infty$), and $X^{N,M}_i$ then remains near $\eta$ with high probability until host $i$ is replaced.   
Also our assumption on the convergence of $X^{N,M}_i(0)$ will imply that $X^{N,M}_i(0)$will be close to $\{0,\eta, 1\}$ with high probability.   
\end{bem}

We postpone the proof of Theorem \ref{TfiniteM} to Section \ref{secProof}, but give here a\\
\textit{Sketch of the proof}: Let $a>0$ and $\epsilon > 0$ be such that Condition ($\mathcal A3$a) formulated in Remark \ref {rem2_3} is valid, and put 
$U^{\eta,N}:= U_a^{\eta,N}$ as in \eqref{Ueta}. We then fix an $\epsilon_1 < \epsilon$ and define
\begin{align}\label{Deta}
 D^{\eta,N} := D_a^{\eta,N}: = [\eta- s_N^{a +\epsilon_1}, \eta + s_N^{a + \epsilon_1}].
 \end{align}
The proof of Theorem \ref{TfiniteM} is based on the following properties of Moran processes subject 
 to balancing selection of moderate strength, which we will derive in a series of lemmata: 
 \begin{itemize}
  \item[a)] (Concentration on the set of states $\{0,1\} \cup U^{\eta,N}$ as $N\to \infty$) \begin{itemize} 
  \item[i)] With high probability, i.e. with a probability tending to 1 as $N\rightarrow \infty$, every host is at any time point,  at which she is involved in a  reinfection
  or host replacement
  event,
  in a state that belongs to the set $\{0,1\} \cup U^{\eta,N}$, see Lemma \ref{aspure}. 
  \item[ii)] With high probability, the interval $U^{\eta,N}$ is left only because of a host replacement event, and not because of the fluctuations that go along
  with the random reproduction
  of parasites, see Lemma \ref{stab}.
  \end{itemize}
  \item[b)]  (Probability of balance) The probability that in a host in state $0$ (i.e. a pure type $B$-host),
  after a reinfection with a single parasite of type $A$, the equilibrium frequency $\eta$ is reached
  before returning to the boundary frequency 0 is $2 \eta s_N + o(s_N)$. Likewise, the  probability that in a pure type $A$-host,
  after a reinfection with a single parasite of type $B$, the equilibrium frequency $\eta$ is reached
  before returning to the boundary frequency 1 is   $2 (1-\eta) s_N + o(s_N)$, see Lemma \ref{balprob}.
  \item[c)]  (Time to balance) The time needed to reach 
  $D^{\eta,N}$ after an effective reinfection
 is  with high probability of order $\mathcal{O}( N^{b (1+a) +\epsilon}/ g_N)$, see Proposition~\ref{TimeToEta}. 
  \end{itemize}
The assumptions  of Theorem \ref{TfiniteM} imply that the parasite frequency in each
host is with high probability contained in $\{0, 1\} \cup D^{\eta,N}$ after a short time.
A host in state 1 or 0 remains in her state until she is (replaced or) hit by a
reinfection event. 
As soon as this host is hit by a reinfection, an \textit{ excursion} of type $A$-parasite frequencies within this host starts which eventually returns
to the starting point or reaches $D^{\eta,N}$ 
before the next reinfection or host replacement event hits this host, according
to property a)i).

If  $D^{\eta, N}$ is reached before the return to the starting point, we call the 
reinfection event {\it effective} and otherwise {\it ineffective}. With $x^{0,N},x^{\eta,N}, x^{1,N}$ denoting the proportions  of hosts with type A-frequencies in $\{0\}$, $\{1\}$  and $U^{\eta,N}$, respectively, it will result from  
property b) and property a)i) that
 the effective reinfection rate of a host in state 0 is $2 \eta (s_N + o(s_N)) r_N (x^{1,N} + (\eta + \mathcal{O}(s_N))  x^{\eta,N} + o(1)).$
 As the interval $U^{\eta,N}$ shrinks to $\{\eta\}$ in the limit $N\rightarrow \infty$, this effective reinfection rate
 converges to $2 \eta r (y^1 + \eta y^\eta)$, with  $y^\ell$ , $\ell \in \{0, \eta, 1\}$, being the proportion of hosts in state $\ell$ in the limit $N\to \infty$.
 Analogously the other effective reinfection rates are obtained.
 Furthermore, as the host replacement rate is 1 (per host), effective
reinfection and host replacement act on the same time scale. According to property c) the transition from $0$ or $1$ to  $D^{\eta,N}$ is 
almost immediate on the host time scale, since $N^{b(1+a) +\epsilon}/ g_N \rightarrow 0$ for $N\rightarrow \infty$.
To show property a)i) we 
will make use of Assumption $(\mathcal A3')$.
%need $g_N \gg N^{b(3 \vee (2+a))+ \epsilon}$. 
The length of a non-effective excursion can be estimated by $N^b+\epsilon/g_N$ with high probability.
Hence the number of reinfection events (occurring at rate $r_N \sim r/s_N = r N^b$)
that hit a host during a non-effective excursion can be estimated by
$N^{2b +\epsilon}/g_N$.
Within a time interval of length $\overline t$ there are with high probability
no more than $c N^b$ reinfection events for some appropriate constant $c$.
Consequently, the number of non-effective excursions hit by an reinfection event 
can be estimated by $c N^{3b+\epsilon}/g_N$. By assumption this number is negligible in the limit $N\rightarrow \infty$.
Similarly one argues for effective reinfections. The length of the transition from $0$ or $1$ to $D^{\eta,N}$
can be estimated by $N^{b(1+a) +\epsilon}/g_N$, and almost surely only a finite number of effective reinfections happen within a bounded time interval.
Hence the number of transitions hit by reinfection events is of order at most
$N^{b(2+a) +\epsilon}/g_N$.

Property a)ii) implies that with high probability the interval $U^{\eta,N}$
is left because of host replacement, and not because of random reproduction of parasites. Hence, only host replacement and effective reinfection drive the evolution of
the limiting system. 

For proving the convergence of $\mathbf{X}^{N,M}$ to $\mathbf{Y}^M$, we will introduce an auxiliary process 
$\hat {\mathbf X}^{M,N}$. For this process the only role of reinfection is to initiate
the transitions from $0$ or $1$ to $D^{\eta,N}$, all further reinfection events are ignored. 
Since reinfection is too weak {\color{black} to lead to essential perturbations of the frequency path
of a transition} from $0$ or $1$ to $D^{\eta,N}$ in $\mathbf X^{M,N}$, the process $\hat{\mathbf X}^{N,M}$
 is a close approximation of ${\bf X}^{N,M}$, in the sense that these processes have the same limiting finite dimensional distributions. Furthermore by using a graphical representation
 for $\hat{\mathbf X}^{N,M}$ we will show in Section \ref{secProof}
 that the finite dimensional distributions of $\hat{\mathbf X}^{N,M}$
 converge to those of $\mathbf Y^{M}$ as $N\to \infty$. Finally, we will prove tightness of  $\mathbf X^{M,N}$ in the $\mathrm M 1$-topology in order to cope with the large number of ineffective
 excursions that are caused by the reinfections.   \\

Theorem \ref{TfiniteM} assumes a finite  host population (of constant size), with each host carrying a large number of parasites.
However, in view of the discussion in Section \ref{Intro}, it is
realistic to assume that also the number of infected hosts is large. We will consider two cases: In Section 2.\ref{secMindepN}
we let first $N\rightarrow \infty$ and then $M\rightarrow \infty$, while in Section 2.\ref{secMdepN} we assume 
a joint convergence of $N$ and $M=M_N$ to $\infty$.

\subsection{Iterative limits: Very large parasite population, large host population}\label{secMindepN}

 In this subsection we focus from the beginning on the processes $\textbf{Y}^M$ with $M$ hosts that arise when the limit $N\to \infty$ of parasite numbers
 per host has been performed according to Theorem \ref{TfiniteM}. Our aim will be to show a``propagation of chaos'' result for $\textbf{Y}^M$ as $M\to \infty$,
 for
 a sequence of initial states that are exchangeable. As a corollary we will obtain that the convergence of the empirical distributions
 of  $\textbf{Y}_0^M$ to the distribution with weights $(v_0^0,v_0^\eta, v_0^1)$ as $M\to \infty$ implies, for each $t>0$, convergence of $\textbf{Y}_t^M$ to the distribution with
 weights $(v_t^0,v_t^\eta, v_t^1)$, given by the solution of the dynamical system
 \begin{align}\label{dynsys}
\dot v^0 & = (1-\eta)v^{\eta} - 2 r \eta v^0 (v^1 + \eta v^{\eta}) \nonumber \\
\dot v^{\eta} & = -v^{\eta} + 2 r (\eta^2  v^0 v^{\eta} + (1-\eta)^2 v^1 v^\eta +  v^0 v^1)\\
\dot v^1 & = \eta v^{\eta} - 2 r (1- \eta) v^1 (v^0 +  (1- \eta) v^\eta) \nonumber
\end{align}
that admits $\Delta^3:= \{(z_0,z_\eta,z_1) \in[0,1]^3 : z_0+z_\eta+z_1=1\}$ as an invariant set.

Proposition \ref{chaos1} will tell that in the limit $M\rightarrow \infty$  the parasite type $A$-frequencies in a typical host perform a $\{0,\eta, 1\}$-valued Markov process  that is defined as follows:

\begin{defi}[Evolution of a typical host in the limit $M\to \infty$]\label{defV}
Write
\[\Delta^3:= \{(z_0,z_\eta,z_1) \in[0,1]^3 : z_0+z_\eta+z_1=1\}.\] In view of Remark \ref{graphYM} and
Proposition \ref{repv}, we define for a given $\mathbf v_0 \in  \Delta^3$  the following time-inhomogeneous
Markovian jump process $(V_t)_{t\ge 0}$ with state space  $\{0,1,\eta\}$: \\
At time $t$ the process  $V$ jumps  from any state to 
 state  
\begin{center}
\begin{tabular}{lll}
 0  & at rate & $v^0_t + (1-\eta) v_t^\eta$ \\
 1  & at rate & $v_t^1 + \eta v_t^\eta$,
\end{tabular}
\end{center}
from state 0  to state  
 $\eta$   at rate  $2 r \eta (v_t^1 +\eta v_t^\eta)$, and\\
from state 1 to state $\eta$  at rate
 $2 r (1-\eta)(v_t^0 + (1-\eta)
 v_t^\eta),$ \\
where  $\mathbf v=(\mathbf v_t)= (v_t^0, v_t^\eta, v_t^1)_{t\ge 0}$ is the solution of the 
 dynamical system \eqref{dynsys} starting from $\mathbf v_0$.
\end{defi}

\begin{prop}[Propagation of chaos]\label{chaos1}
  Assume  \[\frac 1M \sum_{i=1}^M \delta_{Y_i^M(0)}\to v_0^0 \delta_0 +v_0^\eta \delta_\eta + v_0^1 \delta_1\]
 in distribution as $M\to \infty$ for some $\mathbf v^0=(v_0^0, v_0^\eta, v_0^1) \in \Delta^3$.
  Moreover, assume that the initial states $Y^M_1(0), ...,$ $Y^M_M(0)$ are exchangeable, i.e.~arise through
  a drawing without replacement from their empirical distribution (given the latter). Then, for each $\overline t > 0$ the random
  paths $Y^M_i = (Y^M_i(t))_{0\le t\le \overline t}$, $i=1,\ldots,M$, of the host states are exchangeable, and  for each $k \in \mathbb N$, 
 \[(Y^M_1, \ldots, Y^M_k) \to (V_1, \ldots, V_k),\]
  in distribution with respect to  the Skorokhod $\mathrm J_1$-topology as $M\to \infty$, where $(V_1, \ldots, V_k)$ are i.i.d.~copies  of the process
  $V=(V(t))_{0\le t\le \overline t}$ specified in Definition \ref{defV}.
\end{prop}

This proposition as well as the subsequent Corollary \ref{det1} will be proved  in Section \ref{proofs}.1.

For a Polish space $S$ and $0 \le \underline t < \overline t <\infty$ we denote by $D([\underline t, \overline t]]; S)$ the space of c\`{a}dl\`{a}g paths on the time interval $[\underline t, \overline t]$ with state space $S$ and by $\mathcal{M}_1(D([\underline t, \overline t]; S))$
the set of probability measures on the Borel $\sigma$-Algebra on $D([\underline t, \overline t]; S)$ endowed with the Skorokhod $\mathrm J_1$-distance.

 \begin{cor}[Empirical distribution of host states as $M\to \infty$]\label{det1} a)
 In the situation of Proposition~\ref{chaos1} the sequence of $\mathcal{M}_1(D([0,\overline t]; \{0, \eta, 1\}))$-valued random variables
 $\nu^M:= \frac 1M \sum_{i=1}^M \delta_{Y_i^M}$ converges in distribution (w.r.t. the weak topology) to 
 $\mathcal L(V)$, \textcolor{black}{where $\mathcal L$ stands for law, and the space $D([0,\overline t]; \{0, \eta, 1\})$ is equipped with the Skorokhod $\mathrm J_1$-topology.}
 
 b) Moreover, the $\Delta^3$-valued process \[(\textbf{Z}^M(t))_{0\leq t \leq \overline t}:= (Z_0^M(t), Z_\eta^M(t), Z_1^M(t))_{0\leq t  \leq t}\]
 of proportions of hosts in states 0,  $\eta$ and  1, i.e.~
\[Z_\ell^M(t)= \frac{\# \{i \in \{1, ..., M\} | Y^M_i(t)=\ell \}}{M} = \frac{1}{M} \sum_{i=1}^{M} \delta_{Y^M_i(t)}(\ell), \qquad \ell =0,\eta, 1,\] converges in distribution 
 (w.r.t. the Skorokhod $\mathrm J_1$-topology) to $({\bf{v}}_t)_{0\leq t\leq \overline t}$, the solution of the dynamical system \eqref{dynsys}.
 \end{cor}

 \begin{bem}\label{ratesZM}
 The process
  $\textbf{Z}^M$ is a Markovian jump process with jumps from $\textbf{z}= (z^0,z^\eta, z^1)$ to
\begin{center}
\begin{tabular}{lll}
$\textbf{z}+ (\frac{1}{M}, -\frac{1}{M}, 0)$ & at rate &  $M z^\eta( z^0 + (1-\eta) z^\eta)  $  \\
$\textbf{z}+ (\frac{1}{M}, 0, -\frac{1}{M})$ & at rate &  $M (z^{\eta}(1-\eta) z^1+ z^0 z^1)$  \\
$\textbf{z}+ (-\frac{1}{M}, 0, \frac{1}{M})$ & at rate &  $M (z^1 z^0 + z^\eta \eta z^0 ) $ \\
$\textbf{z}+ (-\frac{1}{M}, \frac{1}{M},0)$ & at rate & $2r M (\eta z^1 + \eta^2 z^\eta) z^0 $ \\ 
$\textbf{z}+ (0, \frac{1}{M}, -\frac{1}{M})$ & at rate & $2r M ((1-\eta) z^0 +  (1-\eta)^2 z^\eta) z^1  $ \\
$\textbf{z}+ (0, -\frac{1}{M}, \frac{1}{M})$ & at rate & $ M  z^\eta (z^1+ \eta z^\eta) $. 
\end{tabular}
\end{center}
Later we will obtain Corollary \ref{det1}b)  by projection from its part a), together with a tightness argument. Here we just note in passing that one may also easily check  by a direct calculation that the generator of $\textbf{Z}^M$ 
converges to the generator of the solution $\bf{v}$ of \eqref{dynsys},  which for any  continuously differentiable function $f: \Delta^3 \rightarrow \mathbbm{R}$ is of the form
\[\mathcal{G} f(v^0, v^\eta, v^1) = \sum_{\ell\in \{0, \eta,1\}} \dot{v}^\ell \frac{\partial f}{\partial v^{\ell}}(v^0, v^\eta, v^1).\]
 \end{bem}

\subsection{Joint limit: $M=M_N \rightarrow \infty$ for $N\rightarrow \infty$} \label{secMdepN}

In analogy to Proposition \ref{chaos1}, propagation of chaos can be shown also in the case of a joint limit of $N$ and $M$ to $\infty$, i.e. 
$M= M_N$ and $M_N \rightarrow \infty$ for $N\rightarrow \infty$. This is the topic of the next theorem. Here and in the following we write 
\begin{align}\label{empdist}
\mu^N:= \frac{1}{M_N} \sum_{i=1}^{M_N} \delta_{X_i^{N,M_N}}; \quad \mu^N_t = \frac{1}{M_N}\sum_{i=1}^{M_N} \delta_{X^{N,M_N}_{i}(t)}
\end{align} 
for the empirical distributions of the system of trajectories $\textbf{X}^{N, M_N}$ and their evaluation at some time
$t\geq0$.

\begin{thm}[Propagation of chaos]\label{MF}
 Let Assumptions $(\mathcal{A})$ be valid. For $M=M_N\to \infty$ as $N\to \infty$,  assume that $\mu_0^N$ converges weakly as $N\to \infty$ to
 a distribution $\pi$ on $\{0\}\cup[\alpha, 1-\alpha]\cup\{1\}$ for some $\alpha > 0$. 
 Moreover, assume that for any $N$ the initial states 
  $ X^{N,M_N}_1(0),...,$ $ X^{N,M_N}_{M_N}(0)$ are exchangeable (i.e. arise as drawings without replacement from their empirical distribution $\mu^N_0$).

  Then, for any $0 <\underline t  <\overline t$  and $k \in \mathbb N$, the processes $X_1^{N,M_N}, \ldots, X_k^{N,M_N}$ converge, as \mbox{$N\to \infty$}, in distribution with
  respect to the Skorokhod $\mathrm M_1$-topology, jointly on the time interval $[\underline t ,\overline t]$ to $k$~i.i.d.~copies
  of the process $V$ specified in Definition \ref{defV}, where the distribution of $V_0$ has the weights $\pi(\{0\}), \pi([\alpha, 1-\alpha]), \pi(\{1\})$.
 \end{thm}

We postpone the proof of this theorem and of the next corollary to Section \ref{proofs}.3.

\begin{cor}\label{EmpDist}
 In the situation of Theorem \ref{MF},  
 \begin{itemize} 
 \item [a)] the sequence of $\mathcal{M}_1(D([\underline t,\overline t]; [0,1])$-valued random variables
 $\mu^N$ converges as $N\to \infty$ in distribution (w.r.t. the weak topology) to 
 $\mathcal L(V)$, where  the space $D([\underline t,\overline t]; [0,1])$ is equipped with the Skorokhod $\mathrm M_1$-topology,
 \item [b)]  for each $t >0$ the sequence of $\mathcal{M}_1([0,1])$-valued random variables
 $\mu^N_t$ converges as $N\to \infty$ in distribution (w.r.t. the weak topology) to $v^0_t\delta_0 + v^\eta_t\delta_\eta+ v^1_t\delta_1$, where $\mathbf v = (v^0, v^\eta, v^1)$ is the solution of \eqref{dynsys} starting in $\pi(\{0\}), \pi([\alpha, 1-\alpha]), \pi(\{1\})$ at time 0.
 \end{itemize}
 \end{cor}

\subsection{Properties of the dynamical system $\bf{v}$}
The following result will be proved in Section \ref{proofs}.2.
\begin{prop}[Equilibria]\label{equilibria}
\textcolor{white}{hh}
\begin{itemize}
\item[\bf A] 
 \begin{itemize}
  \item [i)]  
 The dynamical system \eqref{dynsys} has the three equilibrium points $(1,0,0)$, $(0,0,1)$
and $\textbf{u}= (u^0,u^\eta,u^1)$ with 
\begin{align}
u^0 & = \frac{ 2 r \eta(1-\eta)^2- (2\eta-1)}{2 r \eta^2 + 4 r^2 \eta^3(1-\eta)} \notag
\\
u^{\eta} & = \frac{4 r^2 \eta^3(1-\eta)^3 - (2\eta-1)^2 (2r\eta(1-\eta)+1)}{ 2 r \eta^2(1-\eta)^2(1+2 r \eta(1-\eta)} \label{dynsyseq}
\\
u^1 & = \frac{2r(1-\eta)\eta^2 + 2 \eta -1}{2 r (1-\eta)^2 + 4 r^2\eta(1-\eta)^3}. \notag
\end{align}
\item[ii)] $\textbf{u} \in \Delta^3$ iff $r \geq \max\{ \frac{ 2 \eta-1}{2 \eta(1-\eta)^2}, \frac{1-2 \eta}{2(1-\eta)\eta^2} \}$.
At equality, for $\eta>1/2$ the point $\textbf{u}$ equals $(0,0,1)$ and
for $\eta< 1/2$ the point $\textbf{u}$ equals $(1,0,0).$
\end{itemize}
\item[\bf B] 
\begin{itemize}
 \item[i)] 
If 
\begin{align}\label{stablecond}
r > \max\{ \frac{ 2 \eta-1}{2 \eta(1-\eta)^2}, \frac{1-2 \eta}{2(1-\eta)\eta^2} \},
\end{align}
then \begin{itemize}
 \item[{\bf a)}] the equilibria $(0,0,1)$ and $(1,0,0)$ are saddle points, and 
 \item[{\bf b)}] the equilibrium $\textbf{u}$ is globally stable on $\Delta^3 \backslash \{(0,0,1)\cup (1,0,0)\}$.
 \end{itemize}
\item[ii)] For $r \leq \max\{ \frac{ 2 \eta-1}{2 \eta(1-\eta)^2}, \frac{1-2 \eta}{2(1-\eta)\eta^2} \}$
in the case $\eta>0.5$ the equilibrium $(0,0,1)$ is globally
stable on $\Delta^3 \backslash (1,0,0)$ and $(1,0,0)$ is a saddle point, in the case $\eta < 0.5$ the equilibrium $(1,0,0)$
is globally 
stable on $\Delta^3 \backslash (0,0,1)$ and $(0,0,1)$ is a saddle point.
\end{itemize}
\end{itemize}
\end{prop}

\begin{bem}
\begin{itemize}
%\item Because of $v_t^\eta = 1-(v_t^0+v_t^1)$, the system \eqref{dynsys}
 \item  Condition \eqref{stablecond} guarantees the existence of a globally stable equilibrium in the interior of  $\Delta^3$. This condition implies that when initially a small, but non-trivial fraction of type $B$-parasites is present, reinfection is strong enough for type $B$-parasites to invade the parasite population and 
 to direct the host state proportions to the stable equilibrium $\bf{u}$ in the limit $N\rightarrow \infty$, $M\rightarrow \infty$,  see Corollary \ref{EmpDist}.
 In this sense, \eqref{stablecond} can be understood as a condition of invasion fitness.

        For finite but large N and M the stability allows a long time coexistence, see also Section \ref{modres}.\ref{poly}. 
 \item See also Figure \ref{FigEq} for an illustration of the system.
\item We have $u^{\eta} \xrightarrow{r\rightarrow \infty} 1$, i.e. in the limit $r\rightarrow \infty$ there are only  hosts in state $\eta$.

If $r=0$, then eventually one parasite type will be lost.
\item For $\eta=\frac{1}{2}$
 \begin{itemize}
 \item[] $u^0= u^1= \frac{1}{2+r}$,
 \item[] $u^{\eta} = \frac{r}{2+r}$.
 \end{itemize}
Furthermore, in equilibrium the probability to draw from the parasite population a parasite of type $A$ is
 $\frac{1}{2+r} + \frac{r}{r+2} \frac{1}{2} = \frac{1}{2}=\eta$, that is the population mean equals the equilibrium frequency~$\eta$. For $\eta \neq \frac 12$, 
 this is not the case. Indeed, the equation  $u^{\eta} \eta + u^1 =\eta$ implies 
 $u^1= \eta(1-u^{\eta}) = \eta(u^0 +u^1)$, which is equivalent to
 $u^1=\frac{\eta}{1-\eta} u_0.$ One checks that this relationship is only valid for 
 $\eta=\frac{1}{2}.$ 
\end{itemize}
 \end{bem}

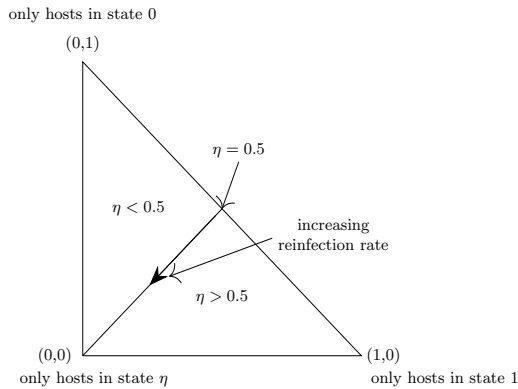
\begin{figure}
\begin{center}
 \resizebox{7cm}{5.2cm}{
\begin{tikzpicture}
 \draw (0,0) -- (5,0) -- (0,5) -- (0,0);
 \node at (.2,-0.4) {only hosts in state $\eta$};
 \node at (-0.5,0) {(0,0)};
 \node at (0,5.8){only hosts in state 0};
 \node at (0,5.3){(0,1)};
 \node at (5.4,0) {(1,0)};
  \node at (6.5,-0.4){only hosts in state 1};
  
\draw (0,0) -- (2.5, 2.5);

\node at (1,2.5) { $\eta < 0.5$};
\node at (2.5,1) { $\eta > 0.5$};

\node at (2.8,3.5) {$\eta=0.5$};
\draw[arrows= -{To[scale=2]}] (2.8,3.3) --(2.5, 2.5);

\draw[-{Stealth[scale=2]}][,shorten >=8pt,shorten <=4.5pt] (2.5, 2.5) -- (1,1);
\node at (4.5, 2.2) {increasing};
\node at (4.5, 1.9) {reinfection rate};
\draw[arrows= -{To[scale=2]}] (3.4, 2)--(1.55,1.35);
 \end{tikzpicture}}
\caption{Position of the coordinates $(u^0, u^1)$ of the stable equilibrium $\textbf{u}$ in the triangle  $\{(x,y): x,y\geq0, x+y\leq 1\}$: 
for $\eta <0.5$, $(u^0, u^1)$ is in the upper-left subtriangle , for $\eta <0.5$, $(u^0, u^1)$ is on the line separating the two subtriangles,
and wanders downwards as $r$ increases. }
\label{FigEq}
\end{center}
 \end{figure}

\subsection{Maintenance of a polymorphic state} \label{poly}

Let Assumptions $(\mathcal{A})$ be fulfilled and assume that the reinfection rate $r$ is not only larger than
$\max\left\{ \frac{2 \eta-1}{2 \eta(1-\eta)^2}, \frac{1-2 \eta}{2(1-\eta)\eta^2} \right\}$ (see Proposition \ref{equilibria} ii)), but even fulfills
\begin{equation}\label{brachcoup}
r>\max\left\{ \frac{\eta}{2 (1-\eta)^2}, \frac{1-\eta}{2 \eta^2}\right\},
\end{equation} 
For large $N$ and $M$, the weights of the empirical frequencies $\mu_t^N$ defined by \eqref{empdist} are close to the solution of the dynamical system \eqref{dynsys} by Corollary \ref{EmpDist} and Lemma \ref{treeV}.   Since the equilibrium~$\textbf{u}$
of  \eqref{dynsys} is stable, then -- once a state close to $\textbf{u}$ is reached -- both parasite types $A$ and $B$
are maintained in the population for a long time. However, because $N$ is finite, eventually one of the types will get lost and the population enters 
a monomorphic state with all hosts being infected either only with type $A$ or only with type $B$. In Theorem \ref{diversity}(ii) we will
give an asymptotic lower bound for this time.

We now enrich our model by allowing, in addition to the rates \eqref{ViralModel}, a two-way mutation for the parasites  at rate $u_N$ per parasite generation. Then 
\begin{equation}\label{popmutrate}
\theta_N:= u_N N M_N g_N
\end{equation}
is the \textit{ population mutation rate}, i.e.~the total rate at which parasites mutate in the total host population on the host time
scale. 
If $\theta_N = o(r_N)$, then (as we will see in the proof of Theorem \ref{diversity}) the  rate at which a type is transmitted by reinfections
is much larger than the mutation rate to that type, even if that type is
retained only in a single host (around 
the equilibrium frequency). In this case the dynamical system which arises as the limiting evolution of $X^{N,M}$ as $N\to \infty$ is not perturbed by mutations. 

Even though most mutations away from a monomorphic population will get lost due to fluctuations, the assumed recurrence of the mutations will eventually turn  a  monomorphic host population into a polymorphic one. In Theorem \ref{diversity}(i) we will give an asymptotic 
upper bound (in terms of $\theta_N$, $s_N$ and $M_N$) on the time at which with high probability, i.e.~with a probability that tends to 1 as $N\rightarrow \infty$, the empirical distribution of the host's states reaches a small neighborhood of $\mathbf u$, when started from a monomorphic state.

A comparison of the two bounds in parts (i) and (ii) of Theorem \ref{diversity} shows that, as long as \mbox{$\theta_N = o(r_N)$} and $\theta_N$ obeys a mild asymptotic lower bound, then 
the proportion of time during which the population is in a monomorphic state is negligible relative to the time during which the population
is in a polymorphic state. We will see that the required lower bound on $\theta_N$ is subexponentially small in the host population size, see Remark \ref{mutRate}(i). 
From a perspective regarding modeling it seems important that the lower bound is this small. Indeed, the polymorphicity we are modeling 
is found in coding regions. Type $A$ and type $B$ represent different genotypes/alleles of the same gene (e.g. in HCMV there exist for the 
region UL 75 two genotypes; these genotypes are separated by one deletion (removing an amino acid) and  8 amino-acid changes, requiring at least 8 non-synonymous point mutations). Since no
``intermediate genotypes'' are found in samples
it is likely that a fitness valley lies between these two genotypes, see \cite{CPBio} for more details
on the biological motivation.

For $\delta>0$ define
\begin{equation}\label{Wdeltau}
W^{\delta,\mathbf u}:= (u^0-\delta, u^0+\delta) \times (u^\eta-\delta, u^\eta+\delta) \times (u^1-\delta, u^1+\delta)
\end{equation}
Furthermore, let
\[\tau^{N}_{\delta, \mathbf u}:= \inf \left\{ t>0 \,| \, (\mu^N_t(\{0\}),\mu^N_t(U^{\eta,N}), \mu^N_t(\{1\}))\in W^{\delta,N} \right\}\]
(recall the definition of $U^{\eta,N}$ in \eqref{Ueta}  together with the explanations in Remark \ref{rem2_3}) and
 let 
\[\tau^N_0 := \inf\{t >0 \, | \,\mu^{N}_t(\{0\})=1 \text{ or } \mu^N_t(\{1\})=1\}\]
be the first time at which the population becomes monomorphic.

 \begin{thm}\label{diversity}
 Let Assumptions $(\mathcal{A})$ and \eqref{brachcoup} be fulfilled, let $\mathbf u$ be as in \eqref{dynsyseq} and let the population mutation rate $\theta$ obey $\theta_N = o(r_N)$ . Choose $\gamma>0$ arbitrarily
small. 
 Then there exists a constant $c_1= c_1(\eta, r)$ such that
\begin{itemize}
\item[(i)] for any $\delta>0$ and any sequence $(\mu_0^N)$ with $\mu_0^N(\{0\})\vee \mu_0^N(\{1\})=1 $, $N\in \mathbb N$, one has\begin{align}\label{MonoToPoly}
  \lim_{N\rightarrow \infty} \mathbbm{P}_{\mu_0^N}(\tau^N_{\delta, \mathbf u} < \frac{1}{c_1 \theta_Ns_N} + M_N^{\gamma} )=1
  \end{align}
\item[(ii)] for any  $\delta < \frac 12 \min\{u^0, u^\eta, u^1\}$ and any sequence $(\mu_0^N)$  with the properties
\[(\mu_0^N(\{0\}), \mu_0^N(U^{\eta,N}), \mu_0^N(\{1\})) \in W^{\delta,\mathbf u}, \quad N\in \mathbb N
\] and  
\[\lim_{N\rightarrow \infty}\mu_0^N(\{0,1\} \cup U^{\eta,N}) = 1\] one has 
\begin{equation}\label{partii}
 \lim_{N\rightarrow \infty} \mathbbm{P}_{\mu_0^N}(\tau^N_0 > \exp( M_N^{1-\gamma})) =1.
\end{equation}
  
 \end{itemize}
\end{thm}

We postpone the proof to Section \ref{proofs}.4.

\begin{bem}\label{mutRate}
\begin{itemize}
\item[i)] From Theorem \ref{diversity} it follows that if $\theta_N^{-1} = o(\exp(M_N^{1-\gamma}) s_N)$ for some $\gamma >0$ (or in other words, if 
 $\theta_N \gg \frac{1}{ \exp(M_N^{1-\gamma}) s_N}$) then most of the time both parasite types coexist in the entire parasite population in a non-negligible amount. 
\item[ii)]The assumption \eqref{brachcoup} on the reinfection rate that is made in Theorem \ref{diversity}  is stronger than \eqref{stablecond}, and allows for a coupling with a supercritical branching process that estimates from below the number of hosts infected with the  currently rare parasite type. The assertion of Theorem~\ref{diversity} might also hold under the weaker condition
\eqref{stablecond}, but we do not have a proof for this.
\end{itemize}
\end{bem}

\section{Proofs}\label{proofs}

We will first deal with the case when the number $N$ of parasites per host is assumed to be infinite: in Section \ref{proofs}.1 we will prove  Proposition  \ref{chaos1}  on the propagation of chaos as the number~$M$ of hosts tends to $\infty$, and in Section \ref{proofs}.2 we will prove  Proposition \ref{equilibria} on the dynamical system that arises in this limit.  In Sections~\ref{proofs}.3 and~\ref{proofs}.4  we will then turn to the limit $N\to\infty $ and prove Theorems \ref{TfiniteM},~\ref{MF} and~\ref{diversity}.

\subsection{Propagation of chaos: Proof of Proposition \ref{chaos1} }\label{Sec3_1}

  Inspired by the graphical representation described in Remark \ref{graphYM} for the process $\mathbf{Y}^M$, we turn right away to a graphical representation for $\mathbf v = (\mathbf{v}_t)_{t\geq0}$, which is the solution of the dynamical
system \eqref{dynsys} (and arises in the limit $M\rightarrow \infty$). This representation (proved in Lemma \ref{repv}) will be in terms of a family of nested trees $(\mathcal T_t)_{t\ge 0}$ {\color{black} where $\mathcal T_t$ depicts all those hosts that potentially influence the state of the host that sits at the root of $\mathcal T_t$, see Figure 4. These trees} will
then be used to prove the ``propagation of chaos'' result Proposition \ref{chaos1}  for the process $\mathbf{Y}^M$ in the limit $M\rightarrow \infty$. {\color{black} As in
Remark \ref{graphYM}, for the two kinds of events that may affect a host  we use again the abbreviations  {\it HR} for host replacement and {\it PER} for potential effective reinfection. When constructing the ``potential ancestry'' of a host at time $t$ one will see that in the limit $M \rightarrow \infty$ whp no collisions occur, i.e. every arrow that hits a line in the potential ancestry at some time $\tau < t$ comes from a host that was not in the potential ancestry between times $\tau$ and $t$. This gives rise to the tree $\mathcal T_t$, as described in the following definition and illustrated by Figure 4.}

\begin{defi}[The labeled random tree $\mathcal T_t$ and the state $C_t$ of its root] \label{graphdyn} Let $v_0^0,v_0^\eta,v_0^1$ be probability weights on $\{0, \eta, 1\}$. For $t\ge 0$
we construct a tree $\mathcal T_t$   with root at time~$t$ and leaves at time $0$ together with a $\{0, \eta, 1\}$-valued random variable $C_t$ as follows (see  Figure 4 for an illustration). A~single (distinguished) line starts from the root backwards
in time. The growth of the tree (backwards in time) is defined via the splitting rates of its lines: 
Each line is hit by 
{\it HR events} at rate 1 and {\it PER events} at rate~$2r$. At each such event, the line  splits into two branches, the {\rm continuing} and the {\rm incoming} one
(where ``incoming'' refers to the direction from the leaves to the root). Whenever the distinguished line is hit by an HR event, we keep both branches in the tree
and designate the continuing branch as the continuation of the distinguished line. Whenever a line other than the distinguished one is hit by a HR  event, we discard
the continuing branch and keep only the incoming one in the tree. At a PER event (irrespective of whether the line is the distinguished one or not) we keep both
the incoming and the continuing branch in the tree. 

Now
assign to the leaves at time 0 independently the states $0$, $\eta$ or $1$ according to the distribution with weights $v_0^0,v_0^\eta,v_0^1$, and let the states
propagate from the leaves up to the root according to the following rule:

At an HR  event (occurring at time $\tau$, say), if the incoming branch at time $\tau-$  is in state $0$ or~$1$,
then the continuing branch  
takes the state of the incoming branch. If the incoming branch at time $\tau-$ is in state $\eta$, then the state of the continuing
branch at time $\tau$ is decided by a coin toss: it takes the state $1$ with probability $\eta$, and the state $0$ with probability $1-\eta$.

At a PER event (occurring at time $\tau$, say) {\color{black} the state of the continuing branch is decided from at most two  independent coin tosses, each with success probability $\eta$, in the following way:} \\
a)  if at  time~$\tau-$ the incoming branch is in state $1$ and the continuing branch is in state $0$ ,  then at time $\tau$ the state of the continuing branch changes  to
$\eta$ with probability $\eta$, and remains in {\color{black} $0$} with probability $1-\eta$, \\ 
b) if at time $\tau-$ the incoming branch is in state $0$  and the continuing branch is in state $1$,  then at time $\tau$ the state of the continuing
branch changes  to $\eta$ with probability $1-\eta$, and
remains in {\color{black}$1$} with probability $\eta$, \\ 
c) 
if at time $\tau-$  the incoming branch is in state $\eta$ and the continuing branch is in state $0$ , then at time $\tau$ the state of the continuing branch changes  to
$\eta$ with probability $\eta^2$, and remains in~$0$ with probability $1-\eta^2$,\\
d) if at time $\tau-$ the incoming branch is in state $\eta$ and the continuing branch is in state $1$, then at time $\tau$ the type of the continuing branch changes  to
$\eta$ with probability $(1-\eta)^2$, and remains in {\color{black}$1$} with  probability $1-(1-\eta)^2$;  \\ 
e) in the remaining cases the continuing branch does not change its state.

In this way, given the tree $\mathcal T_t$ and the realisations of the coin tosses, the states of the leaves are propagated in a deterministic way into the state of the root, which we denote by $C_t$. 
\end{defi}

\begin{bem}\label{possER}
\begin{itemize}
\item[(i)]  \textcolor{black}{Here is a brief explanation of the role of the independent coin tosses with success probability $\eta$ that appear in Definition \ref{graphdyn}.
At a HR event for which the incoming branch is of state $\eta$, such a coin toss decides whether the replacing host has type $1$ (with probability $\eta$) or~$0$.  At a PER event for which the incoming branch is in state $\eta$, a first coin toss decides which type is transmitted, and a second coin toss decides if the reinfection is effective. This second coin toss decreases
the rate $2r$ of potential effective 
reinfection events  to the host-state dependent rate of effective reinfection events. } 
\item[(ii)] Let $\mathcal T_t$ and  $C_t$ be as in Definition \ref{graphdyn}. For given $\overline t > 0$ we can couple the trees $\mathcal T_t$ into a family $(\mathcal T_t)_{0\leq t \leq \overline t}$ of  \textit{ nested} trees, where for any $t<\overline t$  the root of  $\mathcal T_t$ is the vertex at time $t$ of the distinguished line of  $\mathcal T_{\overline t}$, see also Figure 4. In this way, we arrive at the stochastic process $(\mathcal T_t, C_t)_{t \ge 0}$.
\end{itemize}
\end{bem}

\begin{figure}
\begin{center}
 \resizebox{10.5cm}{15cm}{
 \begin{tikzpicture}
 \draw [line width = 0.2mm] (0,0) -- (0,10);
\node at (1.75,10) {\small $C_{\overline t} =\eta$};
\draw [->] (1,10)-- (0.01,10);
 \draw [->] (2,8.5)-- (0.01,8.5);
 \node at (3.5, 8.5) {\small distinguished line};
 \draw  (-4,0) -- (-4, 10);
 \draw (-4, 10) -- (-4.1, 10);
 \node at (-4.4,0) {0};
 \node at (-4.4,10) {$\overline t$};
  \node at (-4,10.5) {\small time};
 \draw (2,0) -- (2,8);
 \draw[->, dashed] (2,8) -- (0,8);
  \node at (2.3, 8) {\tiny $B$,1};
 \draw (-3,0) -- (-3,6);
 \draw[->] (-3,6) -- (0,6);
 \draw [line width = 0.4mm] (1.5, 0) -- (1.5,4);
 \draw[->] (1.5,4) -- (0,4);
   \node at (1.7, 4) {\tiny $B$};
 \draw (-2.2,0) -- (-2.2,2);
 \draw[->, dashed] (-2.2,2)-- (-3,2);
   \node at (-1.85, 2) {\tiny $A$,1 };
 \draw [line width = 0.4mm] (1,0)--(1,2.4);
 \draw[->, dashed] (1, 2.4) -- (1.5,2.4);
 \draw[line width = 0.4mm] (-1,0)-- (-1, 5);
 \draw[->, dashed] (-1,5) -- (0,5);
 \draw (-4,5.3)-- (-4.1,5.3);
 \node at (-4.4, 5.3) {$t$};
 \node at (1.2, 5.3) {$C_t=0$};
  \draw[line width = 0.4mm] (0,5.3)-- (0, 0);
 \draw[->] (0.5, 5.3) -- (0, 5.3);
 \node at (-1.3, 5) {\tiny $B$,0};
 \node at (2,2) {\textbullet};
 
 \node at (2,5.8) {\textbullet};
 
 \node at (-3, 3.6) {\textbullet};
 \node at (-3.2,3.6) {\tiny $A$};
 \node at (-1, 1) {\textbullet};

 \node at (-3, -.4) {0};
  \node at (-2.2, -.4) {$\eta$};
   \node at (-1, -.4) {0};
    \node at (1, -.4) {1};
       \node at (1.5, -.4) {$\eta$};
          \node at (2, -.4) {0};
          \node at (0, -.4){1};
 
 \node at (0.2, 4.6) {0};
% \node at (0.1, 5.6) {0};
 \node at (0.2, 7) {1};
 \node at (0.2, 9) {$\eta$};
 %\node at 
 
 \end{tikzpicture}
 }
 \end{center}
\caption{An example of the tree $\mathcal T_{\overline t}$ as specified in Definition \ref{graphdyn}.
If the distinguished line is hit by an HR event (solid arrows), then both lines are followed downwards,
the incoming (with arrow) and the continuing branch. If another line (different from the distinguished one) is hit by an HR event, then only
the incoming branch is followed further (and a dot is drawn to indicate that an HR event happened). 
Next to the arrows and dots the incoming types and the results of the coin tosses  are recorded {\color{black}like in Figure \ref{graphY}}.
Next to a PER event (dashed arrows) the letter gives the transmitted type ($A$ or $B$)
and the digit indicates the result of the second coin toss, which decides
if the ``potential effective reinfection event'' is realized or not, see Remark \ref{possER}(ii). At an HR event the letter gives the
transmitted type.  At time 0 the lines are coloured according to $\textbf{v}_0$. The state of the distinguished line is displayed between times $0$ and $t$.
The thick lines indicate the branches of the tree $\mathcal T_t$ embedded in the tree $\mathcal T_{\overline t}$.
}
\label{Fig4}
 \end{figure}
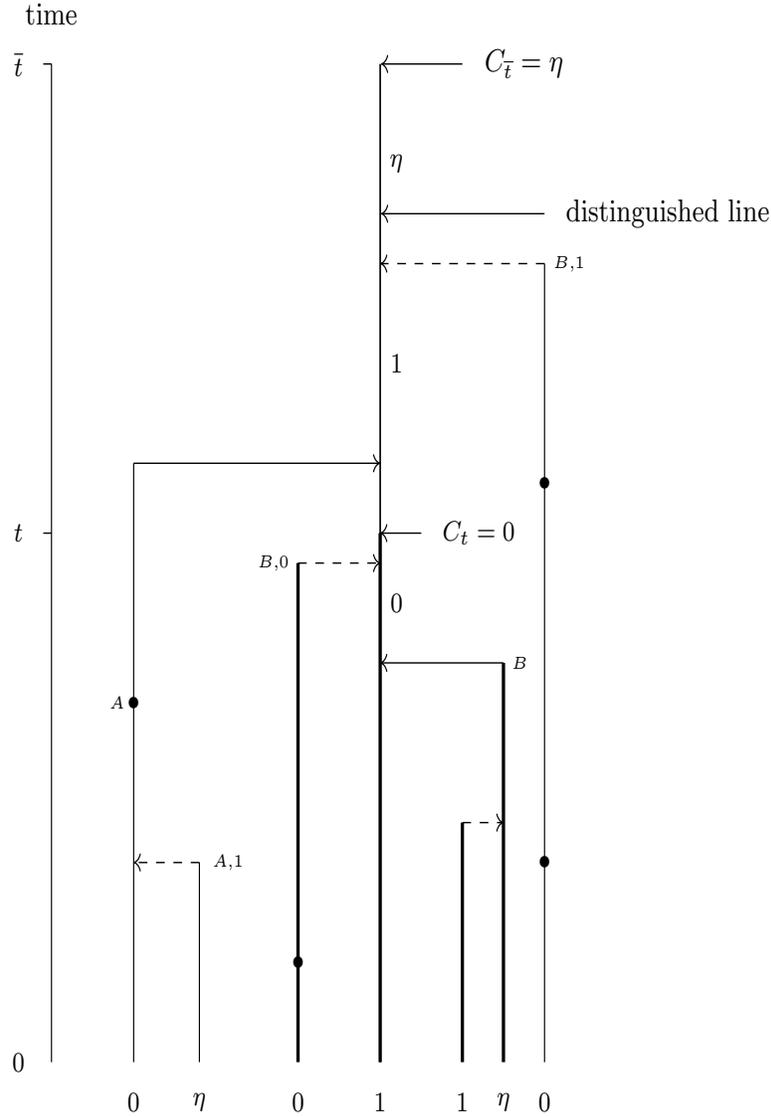

\begin{lem}[Probabilistic representation of the dynamical system \eqref{dynsys}]\label{repv} Let $\mathbf v_0 = (v_0^0,v_0^\eta,v_0^1) \in~\Delta^3$, and let $(\mathcal T_t, C_t)_{t\ge 0}$  be as in Remark \ref{possER} (ii).
The solution $\bf{v}$ of the 
 dynamical system \eqref{dynsys} then has the probabilistic representation 
\[v_t^\ell = \mathbb P(C_t = \ell), \quad t \ge 0, \, \, \ell\in \{0,1,\eta\}.\]
 \end{lem}
\begin{proof} 
We abbreviate
$\mathbf f(t):= (f^0(t), f^\eta(t), f^1(t)):= (\mathbbm{P}(C_t= 0), \mathbbm{P}(C_t=\eta), \mathbbm{P}(C_t=1))$. Then, by construction, $\mathbf f(0) = \mathbf v_0$. 
It thus remains to show that $\mathbf f$ solves the differential equation \eqref{dynsys}.

We check only the equation for the first component, i.e. show that 
\[\frac{ \partial f^0(t)}{\partial t} = (1-\eta) f^\eta(t) - 2 r \eta f^0(t) (f^1(t) + \eta f^\eta(t)), \quad t\in[0,\overline t).\]
(The remaining cases are checked analogously.)

Write
\begin{align*}
\mathbb P(C_{t+\delta} =0)& = \mathbb P(C_{t+\delta} =0 | C_{t}=0)\mathbb P(C_{t}=0)
+ \mathbb P(C_{{t}+\delta} =0 | C_{t}=\eta)\mathbb P(C_{t} =\eta) \\ & \quad \quad + \mathbb P(C_{{t}+\delta} =0 |C_{t}=1)\mathbb P(C_{t}=1) \\
                      & = \mathbb P(C_{{t}+\delta} =0 | C_{t}=0)f^0(t) + \mathbb P(C_{{t}+\delta} =0 | C_{t}=\eta)f^\eta(t) 
                      + \mathbb P(C_{{t}+\delta} =0 |C_{t}=1)f^1(t)
\end{align*}
and calculate
\begin{align*}
\mathbb P(C_{{t}+\delta} =0 | C_{t}=0) = & 1- (2r +1) \delta + \delta (f^0(t) + (1-\eta) f^\eta(t))  \\ & +
2r \delta (f^0(t) + f^1(t) (1-\eta) + f^\eta(t) ((1-\eta) + \eta(1-\eta) ) + \mathcal O(\delta^2)
\end{align*}
To see the latter equality we note first that in the time interval $[t,t+\delta]$ the distinguished line is hit by no more than one (HR or PER) arrow up to
an event of probability $\mathcal O(\delta^2)$.  Then, given $C_t=0$, the root of $\mathcal T_{t+\delta}$ is in state 0 if the distinguished line is not hit by a PER or HR
event between times $t$ and $t+\delta$,
or if an HR event happens but the incoming line is also of type 0 or the incoming line is of type $\eta$ and transmits
type 0, or if a PER event happens but it is not becoming effective, see Remark \ref{possER} (ii).

Similarly we obtain
\[\mathbb P(C_{t+\delta} =0 | C_{t}=\eta)= \delta (f^0(t) + (1-\eta) f^\eta(t)) + \mathcal O(\delta^2)\]
and
\[\mathbb P(C_{{t}+\delta} =0 | C_{t}=1) = \delta (f^0(t) + (1-\eta) f^\eta(t)) + \mathcal O(\delta^2).\]

This leads to
\begin{align*}
&\frac{ \partial f^0(t)}{\partial t}= \lim_{\delta \to 0} \frac{\mathbb P(C_{{t}+\delta}=0) - \mathbb P(C_{{t}}=0)}{\delta}= 
 (1-\eta)f^{\eta}(t) -2 r \eta f^0(t)(f^1(t) +\eta f^\eta(t)).
\end{align*}
\end{proof}
The following corollary is now immediate from Lemma \ref{repv} and Definition  \ref{defV}.
\begin{cor}[Tree representation of the process $V$]\label{treeV}
Under the assumptions of  Lemma~\ref{repv}, the process $(C_t)_{t\ge 0}$ has the same distribution as the process $V$ specified in Definition \ref{defV}. In particular, for $t\ge 0$  the law of
$V_t$ equals $v_t^0 \delta_0 +v_t^\eta \delta_\eta + v_t^1 \delta_1$, where $\mathbf v = (v^0, v^\eta, v^1)$ is the solution of the 
 dynamical system~\eqref{dynsys} with initial condition  $\mathbf v_0$.
\end{cor}

\begin{proof}[Proof of Proposition \ref{chaos1}]
Fix $k \ge 1$ and consider the graphical representation of the processes $Y_1^M, \ldots, Y_k^M  $ described
in Remark~\ref{graphYM}. In this representation, for  fixed $\overline t > 0$, and all $1\le i \le k$,  we trace backwards through the time interval $[0,\overline t]$ the ancestry 
of all those hosts which are potentially relevant for determining the state $Y^M_i(\overline t)$, thus obtaining the \textit{ graph of potential ancestral lineages of host $i$ back from time $\overline t$}. 
Let $E_M$ be the event that \\ 
a) these $k$ graphs are collision-free, in the sense
that they do not share any lines,  \\ and \\
b) that each of these $k$ graphs is a tree.

The probability of the event $E_M$ converges
to 1 as $M\to \infty$. \textcolor{black}{Indeed, each line is hit by an event at a rate bounded by $r+1$. Starting from each time at which a line is hit by an event, we have to follow (backwards into the past) an additional line, which is chosen  randomly out of the $M$ lines according to the corresponding HR or PER event. Thus the number of lines in a single graph grows (at most) like a Yule-tree.
The number of lines in a Yule tree at any fixed time is finite a.s., hence also the total number of potential ancestral lineages in the $k$ graphs is bounded in probability, uniformly for for all times $t \in [0,\overline t] $. Thus the event that at least one of the HR or PER events in the time interval $[0,\overline t]$ produces a collision either within or between one of the $k$ ancestral graphs tends to $0$ as $M\to \infty$. } 
 On the event $E_M$ the graphical representation can be coupled with that of $k$ i.i.d. copies $\mathcal T_{\overline t}^{(1)},\ldots, \mathcal T_{\overline t}^{(k)} $ of the tree $\mathcal T_{\overline t}$ specified in Definition \ref{graphdyn}, and the random
marking of the leaves of this forest results through random draws (without replacement) from the type frequencies
of  $\mathbf Y^M(0)$. Their joint distribution converges by assumption to i.i.d. draws, each with distribution $\mathbf v^0$. 
Thus, on an event of probability arbitrarily close to 1, for sufficiently large $M$, the process $(Y_1^M(s), \ldots,  Y_k^M(t))_{0\le t\le \overline t}$ can be coupled with the process $(C_t^{(1)}, \ldots, C_t^{(k)}) _{0\le t\le \overline t}$, where  $(C_t^{(i)})_{0\le t\le \overline t}$, $i=1, \ldots, k$   are independent copies of the process $(C_t)_{0\le t\le \overline t}$  described in Definition \ref{defV}. The assertion of Proposition \ref{chaos1} now is immediate from  Corollary \ref{treeV}.
\end{proof}

 \begin{proof}[Proof of Corollary \ref{det1}]
a) Since a finite number of draws (with replacement) from a large urn produces no
 collision with high probability, we observe for all $k \in \mathbb N$ and each bounded \mbox{$\mathrm J_1$-continuous} function $f$ defined on
 $D([0,\overline t]; \{0,\eta, 1\})$  that
 \[\left| \mathbbm{E}\left[ \int \ldots \int f(y_1)\cdots f(y_k) \nu^M(dy_1)\ldots \nu^M(dy_k) \right]
 - \mathbb E[ f(Y_1^M)\cdots  f(Y_k^M)] \right|\to 0\]
 as $M\to \infty$. By Proposition \ref{chaos1} the r.h.s converges to $(\mathbb E [f(V)])^k$, hence $\mathbb E[(\int f(y)\nu^M(dy))^k]
 \to (\mathbb E[f(V)])^k$,
 which suffices to conclude the convergence of $\nu_M$ to $\mathcal L(V)$ in the weak topology on $\mathcal{M}_1(D([0,\overline t]; \{0, \eta, 1\}))$, where the latter space is equipped with the $\mathrm J_1$ topology.

  b) The latter convergence  together with Corollary \ref{treeV} implies that for each $t \in [0,\overline t]$
  \[\frac 1M \sum_{i=1}^M \delta_{Y_i^M(t)}\to \mathcal L(V_t)= v_t^0 \delta_0 +v_t^\eta \delta_\eta + v_t^1 \delta_1\]
   in distribution as $M\to \infty$. In other words, writing 
 \[\frac 1M \sum_{i=1}^M \delta_{Y_i^M(t)} = Z_0^M(t) \delta_0 + Z_\eta^M(t) \delta_\eta + Z_1^M(t) \delta_1,\]
 we obtain for the hosts' state frequencies $\textbf{Z}^M(t)= (Z_0^M(t), Z_\eta^M(t), Z_1^M(t)) $ that
\[\textbf{Z}^M(t) \to \mathbf v_t\]
 in distribution as $M\to \infty$.

  The tightness  of $(\mathbf Z^M)_{M\in \mathbb N}:= {(\textbf{Z}^M(t))}_{0 \leq t \leq \overline t}$ with respect to the Skorohkod $\mathrm J_1$-topology can be seen as follows:
  According to the criterion of Theorem 3.7.2 in \cite{EthierKurtz1986} it suffices to show that for any $\delta_1>0$ there exists a $\delta_2>0$ such that 
 \begin{equation}\label{EKcrit}
 \sup_{M\in \mathbbm{N}} \mathbbm{P}(\tilde \omega(\mathbf Z^M, \delta_2) > \delta_1) <\delta_1,
 \end{equation}
 where $\tilde \omega(\mathbf Z^M, \delta):= \max_{k\ge 0}\sup_{(k\delta) \wedge \overline t \le t'\le t''\le ((k+1)\delta)\wedge \overline t} ||\mathbf Z^M_{t''}-\mathbf Z^M_{t'}||$ is an upper bound for the modulus of continuity with resolution $\delta$ of the path $t\to \mathbf Z^M_t$, $0\le t\le \overline t$, and $||\cdot||$ is the variation norm on the set of probability measures on $\{0,\eta,1\}$.
 
 Whenever a host is effectively reinfected or a host replacement occurs, this leads to a jump of~$\mathbf Z^M$ of size at most $1/M$ in the variation norm. \textcolor{black}{ The rate at which a single host is hit by effective reinfection or host replacement can be estimated by some constant $c_3 >0$. Hence the number of events happening on the time interval $I_k:= [k\delta_2\wedge \overline t, (k+1)\delta_2\wedge \overline t]$ can be bounded from above  in pobability    by a Poisson$( \delta_2 c_3 M)$-distributed random variable.   Thus by Chebyshev's inequality}, the distance $||\mathbf Z^M_{t''}-\mathbf Z^M_{t'}||$ is, uniformly in $t', t'' \in I_k$, bounded from above by
\[c_3 \delta_2 + \frac{1}{M^{1/4}}\] with probability at least $1 - \frac{c_3 \delta_2}{M^{1/2}}.$
 Consequently the probability, that in all $\lceil \overline t/\delta_2\rceil $ intervals of length $\delta_2$ the jump size is not larger than $c_3 \delta_2  + \frac{1}{M^{1/4}}$ can be estimated by 
 $\left( 1 -\frac{c_3\delta_2}{M^{1/2}} \right)^{\lceil \overline t/\delta_2\rceil}.$
 For $\delta_2$ small enough and $M> M_0$  with $M_0$ large enough we can achieve that $c_3 \delta_2  + \frac{1}{M^{1/4}} < \delta_1$ and $\left( 1 -\frac{c_3\delta_2}{M^{1/2}} \right)^{\overline t/\delta_2}> 1 -\delta_1.$ This gives \eqref{EKcrit} with  $\sup_{M> M_0}$ in place of $ \sup_{M\in \mathbbm{N}}$.
 
 On the other hand, for $M \leq M_0$ the number of the events happening in the time interval $[0,\overline t]$ can be estimated by a sufficiently large number $K_0$ with probability $p$ arbitrarily close to 1. By choosing $\delta_2$ small enough, one achieves that these events fall into distinct  time intervals $I_k$
 with probability $1 -\delta_1$. For this $\delta_2$ we have
% \textcolor{black}{In other words for $M \leq M_0$
% we have for $\delta_2$ sufficiently small}
{\color{black}\[ \sup_{M\le M_0}\mathbbm{P}(\tilde \omega(\mathbf Z^M, \delta_2) > 0) <\delta_1.\]}
Altogether this shows \eqref{EKcrit} and proves the claimed tightness.
 \end{proof}

\subsection{Proof of Proposition \ref{equilibria}}\label{Sec3_2}
\textbf{A} (i)
It is obvious from \eqref{dynsys} that (0,0,1) and (1,0,0) are equilibrium points.\\
Further, we calculate from $\dot u^0=0$ and $\dot u^{1} =0$, that 
\[ u^{\eta}= \frac{ 2 r \eta u^0 u^1}{1-\eta - 2 r \eta^2 u^0}\]
and
\[u^1= \frac{2\eta-1 + 2 r\eta^2 (1-\eta) u^0}{2 r (1-\eta)^2 \eta}\]

$\dot u^{\eta}=0$ is then automatic, since $\dot u^1 + \dot u^0 = \dot u^{\eta}$.

Using $u^0+u^\eta + u^1=1$ 
we obtain that
\[u^0= \frac{ 2 r \eta(1-\eta)^2- (2\eta-1)}{2 r \eta^2 + 4 r^2 \eta^3(1-\eta))},\]
 hence
\[u^1= \frac{2\eta-1 + 2 r\eta^2 (1-\eta) u^0}{2 r (1-\eta)^2 \eta} =\frac{2r(1-\eta)\eta^2 + 2 \eta -1}{2 r (1-\eta)^2 + 4 r^2\eta(1-\eta)^3}\]
and 
\begin{align*}
u^{\eta} & = \frac{ 2 r \eta u^0 u^1}{1-\eta - 2 r \eta^2 u^0} 
= \frac{u^0(2\eta-1 + 2r\eta^2(1-\eta)u^0)}{(1-\eta)^2(1-\eta- 2 r \eta^2 u^0)} \\ & =
\frac{4 r^2 \eta^3(1-\eta)^3 - (2\eta-1)^2 (2r\eta(1-\eta)+1)}{ 2 r \eta^2(1-\eta)^2(1+2 r \eta(1-\eta)}.
\end{align*}

(ii) If $r= \frac{2 \eta -1}{ 2\eta (1-\eta)^2}$, one calculates that $u^0=u^\eta =0$ and $u^1=1.$
For $r< \frac{2 \eta -1}{ 2\eta (1-\eta)^2}$ the equilibrium point $\textbf{u}$ does not belong to $\Delta^3$.

If $r > \max\{ \frac{2 \eta-1}{2\eta (1-\eta)^2}, \frac{1-2\eta}{2 \eta^2 (1-\eta) } \}$, the equilibrium
point $\textbf{u}$ lies in the interior of $\Delta^3.$
This can be seen as follows: First for $r > \max\{\frac{2 \eta-1}{2\eta (1-\eta)^2}, \frac{1-2\eta}{2 \eta^2 (1-\eta) } \}$ the components 
$u^0$  as well as $u^1$ are strictly positive.  
For $\eta <  \frac 1 2$ and  $r= \frac{1- 2\eta}{ 2 \eta^2 (1-\eta)}$, we have $u^1=0$ and $u^0=1$. For $r\rightarrow \infty$ both $u^0 \rightarrow 0$ and
$u^1\rightarrow 0$. We calculate that $\frac{\partial u^0}{\partial r} =  - 8r^2 \eta^4(1-\eta)^3 + 8 r \eta^3 (1-\eta)(2\eta-1) + 2\eta^2(2\eta-1).$
For $r >0$ this derivative is negative, as it is negative for $r\rightarrow \infty$ and both roots have negative real parts according
to the Routh-Hurwitz
criterion, see \cite{GradshteynRyzhik}, \S 15.715, p. 1076. Consequently $u^0$ is strictly monotonically decreasing from 1 to 0.
Analogously one argues for $u^1$ and $u^\eta$.\\
\textbf{B} (Stability)
To analyze the stability of the dynamical system $\textbf{\textit{v}}$ we project the system onto 
$\myco = \{(x,y) | x,y \in[0,1], x+y \leq 1\}$
by considering only the coordinates $v^0$ and $v^1$. The three fixed points correspond in this projection to
the points $(0,1), (1,0)$ and $(u^0, u^1)$.

We recall that an equilibrium point is asymptotically stable if all real parts
of the eigenvalues of the Jacobian are strictly negative. 
 
The Jacobian $\mathcal{J}_v$ of the dynamical system
$\textbf{\textit{v}}$ is
 \[\mathcal{J}_v = \begin{pmatrix} -(1-\eta) - 2 r \eta (v^1 (1-\eta) +\eta) + 4 r \eta^2 v^0 & -(1-\eta) - 2 r \eta v^0(1-\eta) \\
 -\eta -  2 r (1-\eta) \eta v^1 & -\eta - 2 r (1-\eta)(v^0 \eta + 1-\eta) + 4 r (1-\eta)^2 v^1.
\end{pmatrix} \]

(i) Case $r > \max\{ \frac{ 2 \eta-1}{2 \eta(1-\eta)^2}, \frac{1-2 \eta}{2(1-\eta)\eta^2} \}$:

{\bf a)} (Equilibrium points at the boundary)\\

For $v^0=0$ and $v^1=1$ the Jacobian equals
$\begin{pmatrix}
 \eta-1-2 r \eta & \eta-1\\
  -\eta -2 r (1-\eta) \eta  & -\eta + 2 r (1-\eta)^2. 
\end{pmatrix}$

In this case the eigenvalues solve
\begin{align}\label{eigenv}
\lambda^2 + \lambda (1 + 2 r (\eta - (1-\eta)^2)) +  2r (\eta^2 -(1-\eta)^3 - (1-\eta)^2 \eta) - 4 r^2 \eta(1-\eta)^2=0.
\end{align}

The second coefficient is $<0$ iff $r > \frac{2 \eta -1}{2\eta(1-\eta)^2 }$. (Analogously for $v^0=1$  and $v^1=0$ the second coefficient is $<0$ iff 
$r > \frac{1 -2 \eta}{2(1-\eta)\eta^2 }$.) Hence, (0,1) is for
$r > \max\{ \frac{ 2 \eta-1}{2 \eta(1-\eta)^2}, \frac{1-2 \eta}{2(1-\eta)\eta^2} \}$
a saddle point. Analogously, one argues that (1,0) is a saddle point. 
 
According to Corollary \ref{det1} the dynamical system $\textbf{\textit{v}}$ is the deterministic limit of the Markov process
$\textbf{Z}^M$, whose state space is a subset of $\Delta^3.$
Consequently, in the points (0,1) and (1,0)
one eigenvector (appropriately oriented) is pointing into the interior of $\Delta^3$ and one
pointing in a direction outside of $\Delta^3$.

 {\bf b)} (Equilibrium point in the interior) \\
 The fixed point $\textbf{u}$ is globally stable in $\myco \backslash\{(0,1), (1,0)\}$, if all trajectories starting 
 in a point of $\myco \backslash\{(0,1), (1,0)\}$ converge to $\textbf{u}$. By the Poincar\'e-Bendixson-Theorem, see e.g. Theorem 1.8.1 in \cite{GuckHolmes}, 
 in a 2-dimensional differential dynamical system each compact $\omega$-limit set, containing only finitely many fixed 
 points, is either a fixed point, a periodic orbit or a connected set, consisting of homoclinic and heteroclinic orbits connecting
 a finite set of fixed points. We can exclude heteroclinic orbits, because the vector field corresponding to the dynamical system 
 $\textbf{\textit{v}}$
 is pointing into the interior of $\Delta^3$.
 So if we can show that no periodic orbits exist in the interior of $\Delta^3$, then the limit set has
 to be a fixed point, the point $\textbf{u}$. We check that the partial derivatives fulfill $\frac{\partial{\dot{v}^i}}{\partial v^j}\leq 0$
 for $i,j \in \{0,1\}$, $i\neq j$ and $v^i,v^j \in \myco$. Hence, the dynamical system is competitive 
 and we can apply Theorem 2.3 of \cite{Hirsch1982} stating that all trajectories starting
 in $\myco\backslash \{(0,1), (1,0)\}$
 converge to some point in $\myco$. This point can only be the fixed point $\textbf{u}$.

(ii) Case $r <  \max\{ \frac{ 2 \eta-1}{2 \eta(1-\eta)^2}, \frac{1-2 \eta}{2(1-\eta)\eta^2} \}$: \\
In this case $\textbf{u}$ does not belong to $\Delta^3.$ To assess the stability of 
the other fixed points consider first the case $\eta > \frac 1 2$:
Then $\max\{ \frac{ 2 \eta-1}{2 \eta(1-\eta)^2}, \frac{1-2 \eta}{2(1-\eta)\eta^2} \} = \frac{2 \eta -1}{\eta(1-\eta)^2 }$
and since $r>0$, we obtain that
$r > \frac{1-2\eta}{2 (1-\eta) \eta^2}$. Consequently, $(1,0)$ is a saddle point.

As $r < \frac{2\eta-1}{\eta}$, the second coefficient of \eqref{eigenv} is $>0$. To assess the stability of 
$(0,1)$ we analyze the first coefficient. For $\eta > \frac{3 -\sqrt{5}}{2}$ the first coefficient is $>0$.
Consequently, both eigenvalues are strictly negative and hence $(0,1)$ is 
an asymptotically stable equilibrium point.

Since no other equilibrium point is contained in $\Delta^3$, the point $(0,1)$ must also be a globally stable equilibrium point on $\Delta^3 \backslash (1,0).$ 

For $\eta < \frac{1}{2}$ one argues analogously.
  
  Case $r= \max\{ \frac{ 2 \eta-1}{2 \eta(1-\eta)^2}, \frac{1-2 \eta}{2(1-\eta)\eta^2}\}$:

The claim follows by continuity. Consider the case $\eta>\tfrac 1 2$:
For $r< \max\{ \frac{ 2 \eta-1}{2 \eta(1-\eta)^2}, \frac{1-2 \eta}{2(1-\eta)\eta^2}\}$
 the point (1,0) is a saddle point and (0,1) is globally stable
on $\Delta^3\backslash (1,0)$.
On the other hand for $r >  \max\{ \frac{ 2 \eta-1}{2 \eta(1-\eta)^2}, \frac{1-2 \eta}{2(1-\eta)\eta^2}\}$ the point (1,0) is a saddle point,
$\textbf{u}$ is globally stable on $\Delta^3\backslash \{(0,1) \cup (1,0) \}$ and 
$\textbf{u} \rightarrow  (0,1)$ as $r \downarrow \max\{ \frac{ 2 \eta-1}{2 \eta(1-\eta)^2}, \frac{1-2 \eta}{2(1-\eta)\eta^2}\}$.
%., as well
%as for $r<\max\{ \frac{ 2 \eta-1}{2 \eta(1-\eta)^2}, \frac{1-2 \eta}{2(1-\eta)\eta^2}\}$ the point 
%(0,1) is globally stable on $\Delta^3 \backslash (1,0)$ and (1, 0) is a saddle point.

One argues analogously for $\eta <  \frac 1 2.$ \hfill $\Box$

\subsection{Proofs of Theorem \ref{TfiniteM} and Theorem \ref{MF}}\label{secProof}

In the following we will use the
phrase \textit{with high probability as $N\to \infty$} or simply \textit{whp} as a synonym for {\it with probability converging to 1 as $N\rightarrow \infty$}.

As observed in Remark \ref{rem2_3}, there exist positive numbers $a$ and $\epsilon$ that satisfy assumption ($\mathcal A3$a). With this $a$  and with an $\epsilon_1 < \epsilon$, we define  $U_a^{\eta,N}$ as in \eqref{Ueta} and $D_a^{\eta,N}$ as in \eqref{Deta}.
In the following we will suppress the superscript $a$ and write just $ U^{\eta,N}$ and $ D^{\eta,N}.$ 
 
As already mentioned in the sketch of the proof of Theorem \ref{TfiniteM} we distinguish between effective and ineffective reinfection events. 
Assume host $i$ is reinfected at time $t>0$, and $X_i^N(t-)$ is in state $0$ or $1$, then $X_i^N$ either returns to that
state before it reaches $D^{\eta,N}$, or
it reaches $D^{\eta,N}$ before returning to that state. In the former case we call the reinfection event {\it ineffective}
and in the latter {\it effective}.

The proofs of Theorems \ref{TfiniteM} and \ref{MF} are based on several lemmata; these we state next. Basically these lemmata make 
statements about hitting probabilities and hitting times of the path of the 
frequency of type $A$ in a single (isolated) host.  

The following lemma is elementary and well-known:
\begin{lem}[Ruin probabilities]\label{ruin}
Let $W$ be a random walk on $\mathbb Z$ starting in $0$, with increment distribution $p\delta_1 + q\delta_{-1}$, $p+q=1$. Then, for $N_1$ and $N_2 \in \mathbb N$, the
probability that $W$ hits $N_1$ before it hits $-N_2$ is
\begin{equation}\label{hitprob}
\frac{\left(p/q\right)^{N_2}-1}{\left(p/q\right)^{N_1+ N_2} -1} =  1-\frac{\left(q/p\right)^{N_1}-1}{\left(q/p\right)^{N_1+ N_2} -1}.
\end{equation}
\end{lem}

\begin{lem}[Probability to balance]\label{balprob}
Let Assumptions $(\mathcal{A})$ be fulfilled.% ,

 Let $\xi^N= (\xi^N_t)_{t\geq0}$ be a Markov process on $\{0,1/N ..., 1\}$ with jumps from
 $i/N$ to 
 \begin{center}
 \begin{tabular}{lll}
  $\frac{i+1}{N}$ & at rate & $g_N\left(i\frac{N-i}{N} (1+ s_N (\eta -\frac{i}{N}))+r'_{N,1}\right)$ \\
  $\frac{i-1}{N}$ & at rate & $g_N\left(i\frac{N-i}{N} (1- s_N (\eta -\frac{i}{N}))+r'_{N,2}\right)$ \\
 \end{tabular}
 \end{center}
 for some $ r'_{N,1},  r'_{N,2}$ which may change in time but satisfy
 \begin{align}\label{rprime}
 0\le r'_{N,1}, r'_{N,2} \le r_N/g_N.
 \end{align}
 
 For $x\in \{0, \frac{1}{N}, ..., \frac{N-1}{N},1\}$ let \[\tau_x = \inf\{ t \geq 0 | \xi^N_t= x \}.\]  
 Then 
 \begin{align}\label{AinB}
 \lim_{N\rightarrow \infty} \frac{ \mathbbm{P}_{\frac{1}{N}}(\tau_{\frac{\lceil \eta N \rceil}{N}} < \tau_0\})}{2 s_N \eta} =1 
 \end{align}
 and 
 \begin{align}\label{BinA}
  \lim_{N\rightarrow \infty} \frac{ \mathbbm{P}_{\frac{N-1}{N}}(\tau_{\frac{\lfloor \eta N \rfloor}{N}} < \tau_N\})}{2 s_N (1-\eta)} =1 ,
  \end{align}
  where the subscript of $\mathbbm{P}$ denotes the initial state of $\xi^N.$ 
  \end{lem}
\begin{proof} First we remark that from Assumption \eqref{rprime} we obtain 
\begin{equation}\label{boundr} 
r'_{N,k} \le \frac {r_N}{g_N} = \frac{r}{s_N g_N} = o(s_N), \quad k=1,2.
\end{equation}
We prove only statement \eqref{AinB}; the companion statement \eqref{BinA} follows analogously.

In order to tie in with Lemma \ref{ruin} we  consider the process $\xi'= N\xi^N$ and show instead of \eqref{AinB} the convergence
 \begin{align}\label{AinBprime}
 \lim_{N\rightarrow \infty} \frac{ \mathbbm{P}_1(\tau'_{\lceil \eta N \rceil/N} < \tau'_0\})}{2 s_N \eta} =1 
 \end{align}
with $\tau'_{x}=\inf\{ t \geq 0 | \xi'_t= x N \}$ for $x\in \{0, \frac{1}{N}, ..., \frac{N-1}{N},1\}$.
  
To prove an upper bound on the probability in the numerator of
\eqref{AinBprime} note that as long as $1\leq  i\leq \eta N$
the probability  that the next event is a birth is 
$\frac{\frac{i(N-i)}{N}(1+  s_N\eta ) +r'_{N,1} }{\frac{2  i(N-i)}{N} + r'_{N,1} + r'_{N,2}} = \frac{1+ s_N\eta + o(s_N)}{2}$
according to \eqref{boundr}.
Hence, we can couple $\xi'$ with an (asymmetric) random walk $(R^{(1)}_n)_{n\geq 0}$, which makes jumps of size one upwards with probability $\frac{1+ s_N\eta + o(s_N)}{2}$ and downwards with probability $\frac{1- s_N\eta + o(s_N)}{2}$, and is absorbed at 0, such that
for any $0<\delta <\eta$
\[\mathbbm{P}_1 (\tau'_{\lceil \eta N \rceil/N} < \tau'_0) \leq \mathbbm{P}_{1}( \exists k \geq 0 : R^{(1)}_k \geq \delta N ). \]

For the random walk $(R^{(1)}_n)_{n\geq 0}$ we have \[\mathbbm{P}_{1}( R^{(1)}_\infty = \infty ) = 1 -   
\mathbbm{P}_{1}( R^{(1)}_\infty = 0 )\] and by Lemma \ref{ruin} we have 
\[\phi_N := \mathbbm{P}_{1}( R^{(1)}_\infty = \infty )= 2 s_N \eta + o(s_N).\]

Furthermore using $s_N N \rightarrow \infty$ we have

\[\lim_{N\rightarrow \infty} 
\frac{\mathbbm{P}_{1}( \exists k \geq 0 : R^{(1)}_k \geq \delta N ) }{
\mathbbm{P}_{1}(  R^{(1)}_\infty = \infty )}  
= 1, \] since
$ \mathbbm{P}(R^{(1)}_\infty = \infty | \exists k\geq 0 : R^{(1)}_k \geq \delta N) \geq 1 - (1- \phi_N)^{\delta N} 
 = 1 - (1- 2 s_N \eta + o(s_N))^{\delta N} \rightarrow 1.$

To obtain a lower bound on the probability in the numerator of \eqref{AinBprime}, we fix an arbitrary $0<\delta <\eta$ and
note that we
can couple $\xi'$ with a random walk $(R^{(2)}_n)_{n\geq0}$ 
which makes jumps of size one upwards with probability
$\frac{1+s_N(\eta -\delta) + o(s_N)}{2} $ and downwards with probability
$ \frac{1 - s_N(\eta- \delta) + o(s_N)}{2}$,
and is absorbed 0, such that
\[\mathbbm{P}_1( R^{(2)}_\infty =\infty) \leq \mathbbm{P}_1(\tau'_{\frac{\lceil\delta N\rceil}N} <\tau'_0). \]

We have (again by Lemma \ref{ruin})  
\[\mathbbm{P}_1( R^{(2)}_\infty =\infty) = 2 s_N (\eta - \delta) + o(s_N).\]

To finish, we show that the probability $p^N_\delta$ that $\xi^N$ hits $\frac{\lceil \eta N \rceil}{N}$  before 0 (when starting in $\frac{\lfloor \delta N \rfloor}{N}$) also tends to~1 for $N\rightarrow \infty$. Since $\delta>0$ was arbitrary this concludes the proof.
To do so we calculate first the probability $\tilde{p}^N_\delta$ that $\xi^N$ hits $\frac{\lceil (\eta- \delta)N \rceil}{N}$  before 0, when starting in~$ \frac{ \lfloor \delta N \rfloor}{N}$. 
This can be estimated by the the hitting probability \eqref{hitprob} with $p= \frac{1+ s_N \delta}{2} + o(s_N)$, 
$N_1= \lceil N(\eta-\delta) \rceil$ and $N_2= \lfloor N\delta\rfloor $.
Then 
\[\tilde{p}^N_\delta \geq 1- \frac{ (1 + 2 s_N \delta + o(s_N))^{\lfloor \delta N \rfloor}-1 }{(1+ 2 s_N \delta + o(s_N))^{\lfloor \eta N \rfloor} -1}\ge 1 - \exp(N^{1-b}( 2 \delta(\delta -\eta)) + o(\exp(N^{1-b}( 2 \delta(\delta -\eta)))
\xrightarrow{N\rightarrow \infty}  1.\]
(This estimate of the speed of convergence towards 1 will be helpful in the proof of Theorem \ref{TfiniteM}.)   For $\delta$ small enough, the probability to hit  $\frac{\lfloor  \eta N \rfloor}{N}$ before $\frac{\lceil \delta N \rceil}{N}$ when starting in $\frac{\lfloor(\eta-\delta)N \rfloor}{ N}$, 
can be estimated from below by $\frac{1}{2}$.
Hence we arrive at 
\begin{align}\label{leveldelta}
p^N_\delta \geq \sum_{k=0}^{\infty} \left(\frac{1}{2}\right)^{k+1} {(\tilde{p}^N_\delta)}^k = \frac{1}{2} \frac{1}{1- \tilde{p}^N_\delta/2} \ge 1 - 2\exp(N^{1-b}( 2 \delta(\delta -\eta)) \xrightarrow{N\rightarrow \infty} 1.
\end{align}

\end{proof}

In analogy to the ineffective reinfections discussed at the beginning of this subsection we will speak of a {\it non-effective excursion from $y$} when a path starting in state $y \in [0,1]$ returns to $y$
before it hits the frequency $\frac{\lfloor \eta N \rfloor}{N} $.

\begin{bem}\label{heightbound}
In the proof of Lemma \ref{balprob} we
showed that 
 $\lim_{N\rightarrow \infty} \mathbbm{P}_{\frac{1}{N}}(\tau_{\frac{\lceil \eta N \rceil }{N}} < \tau_0 | \tau_{\frac{\lceil \delta N \rceil}{ N}} < \tau_0) =1, $ for any $\delta>0$.
 This implies that 
 there exists a sequence $\delta_N$ converging to 0 sufficiently slowly such that 
  \[\lim_{N\rightarrow \infty} \mathbbm{P}_{\frac{1}{N}}(\tau_{\frac{\lceil \eta N \rceil }{N}} < \tau_0 \, \mid \,  \tau_{\frac{\lceil \delta_N N \rceil}{ N}} < \tau_0) =1 .\]
 We can interpret the last statement also as a bound on the height of a non-effective excursion from~0. Specifically, we have
 \[\lim_{N\rightarrow \infty} \mathbbm{P}_{\frac{1}{N}} (\max_{0<t<\tau_0} \xi^N_t < \frac{ \lceil \delta_N N \rceil}{N}\, \big | \, \tau_0 <\tau_{\frac{\lceil \eta N \rceil }{N}}) =1.\]
    \end{bem}

For the proofs of Theorem \ref{TfiniteM} and Theorem \ref{MF} we need an estimate on the time that a type~$A$-frequency path  needs to reach 
the  interval $D^{\eta, N}$, see Equation \eqref{Deta}, when starting from $\frac{1}{N}$ or $1 -\frac{1}{N}$, respectively.
This as well as an estimate on the asymptotic time to eventually reach the equilibrium frequency $\eta$  will be handled in the next proposition.

\begin{prop}[Time to balance]\label{TimeToEta}
Let Assumptions $(\mathcal{A}1)$, $(\mathcal{A}2)$ and $(\mathcal{A}3')$ be fulfilled. Let $\xi^N$ be a Markov process as in Lemma \ref{balprob}. For any $\epsilon > \epsilon_1 >0$ we have as $N\to \infty$ 
\begin{align}\label{firsteq}
\mathbbm{P}_{\frac{1}{N}}( \tau_{  \frac{\lceil (\eta -s_N^{a +\epsilon_1})N \rceil}{N}  } &< \frac{ N^{b(1+a) +\epsilon}}{g_N} \, \mid \, \tau_{ \frac{\lceil (\eta -s_N^{a+\epsilon_1})N \rceil}{N}} <\infty) \rightarrow 1,\\ \label{secondeq}
\mathbbm{P}_{\frac{1}{N}}( \tau_{\frac{\lceil \eta N \rceil}{N}} &<  \frac{N^{\frac{(1+2b)}{3} + \epsilon}}{g_N}\,  \mid \,
  \tau_{\frac{\lceil \eta N \rceil}{N}} <\infty) \rightarrow 1.\end{align} Analogous statements hold when the process is started in $\frac{N-1}{N}$.
  \end{prop}

To prepare the proof we recall some well-known facts on the first and second moments of exit times of simple random walks.

\begin{lem}\label{rw}
Let $h_N >\frac{1}{N}$ and $(W_t^{(N)})_{t \geq 0}= \frac{S_{\lfloor N^2t \rfloor}}{N}$ be a
rescaled, symmetric random walk with $S_j= \sum_{k=1}^j \zeta_k$ for iid $(\zeta_k)_{k\geq 1}$
with $\mathbbm{P}(\zeta_1=1)= \mathbbm{P}(\zeta_1=-1)=\frac{1}{2}$. Let 
\[T_{h_N} = \inf\{t\geq 0 | | W_t^{(N)} |\geq h_N\}. \]
Then 
 \[\mathbbm{E}[T_{h_N}]= h_N^2\]
 as well as

\[ \mathbbm{E}[T^2_{h_N}]= \frac{5 h_N^4}{3} - \frac{2 h_N^2}{3 N^2},\]
from which follows
\[\mathbbm{V}[T_{h_N}]= \frac{2}{3}(h_N^4 - \frac{ h^2_N}{N^2}).\]
\end{lem}

\begin{proof}
For the unscaled random walk $(S_j)_{j\geq1}$ it is well-known  that the expected hitting time of the set $\{h_N N, -h_N N\} $
is $(h_N N)^2$. The second moment of this hitting time can be calculated by considering the martingale 
\[M_j  := S_j^4-6jS_j^2+3j^2+2j, \quad j=0,1,\ldots\]
Rescaling proves the claim.
\end{proof}

\begin{proof}[Proof of Proposition \ref{TimeToEta}]% Let us write $\xi:= \xi^N$ for brevity.
We have  $\{\tau_{\frac{\lceil \eta N\rceil}{N}}<\infty\} \subset
\{\tau_{\frac{\lceil( \eta - s_N^{a +\epsilon_1}) N\rceil}{N}}<\infty\}$ and 
\[\lim_{N\rightarrow \infty} \frac{\mathbbm{P}_1(\tau_{\frac{\lceil( \eta - s_N^{a +\epsilon_1}) N\rceil}{N}}<\infty)}{\mathbbm{P}_1(\tau_{\frac{\lceil \eta N\rceil}{N}}<\infty) }=1.\]
Hence, for proving \eqref{firsteq} it suffices to condition on the event $\{\tau_{\frac{\lceil \eta N\rceil}{N}}<\infty\}$.

 We separate the time to reach the frequency $\frac{\lceil \eta N \rceil}{N}$ starting  from the frequency $\frac{1}{N}$  into four phases:
 \begin{itemize}
 \item[] 1) Reaching a (fixed) frequency $\frac{\lceil h N \rceil}{N} >0$, for some $h>0$,
 \item[] 2) Climbing from $\frac{\lceil h N \rceil}{N}$ to $\frac{\lceil ( \eta-h)N \rceil}{N} $, 
\item[] 3) Climbing from $\frac{\lceil( \eta -h) \rceil}{N}$ to $\frac{\lceil (\eta-h_N)N \rceil}{N}$ for some appropriate sequence $h_N$ with $h_N \rightarrow 0$,
 \item[] 4) Reaching $\frac{\lceil\eta N \rceil}{N}$ from $\frac{\lceil (\eta-h_N)N \rceil}{N}$.
 \end{itemize}
 For the proof of \eqref{firsteq} only the first three phases are relevant.

 For a stochastic process $\mathcal{H}= (H_t)_{t\geq 0}$ we put (with a slight abuse of notation)
 \[ T^{\mathcal{H}}_{h} := \begin{cases}
  \inf\{t \geq 0 | H_t \geq  \frac{\lceil  hN  \rceil}{N} \} \mbox{ if } \mathcal{H}  \mbox{ is  } [0,1]\mbox{-valued, } \\
  \inf\{t \geq 0 | H_t \geq \lceil  hN  \rceil \} \mbox{ if } \mathcal{H}  \mbox{ is  }  \{0, ..., N\}\mbox{-valued. }
  \end{cases}
  \]
   
 In phase 1, we consider as in the proof of Lemma \ref{balprob} the process $\xi' = N \xi^N$ instead of $\xi^N$. The process $\xi'$ is a birth-death process 
 with birth rate 
$\lambda_i=  i\frac{N-i }{N} g_N (1+ s_N (\eta - \frac{i}{N}) + o(s_N))$
and death rate $\mu_i = i\frac{N-i }{N} g_N (1 - s_N (\eta - \frac{i}{N}) + o(s_N))$, according to \eqref{boundr}. 

 Note
that $\lambda_i \geq i\frac{N-i }{N} g_N (1+ s_N (\eta - h) + o(s_N))$ and 
$\mu_i  \leq i \frac{N-i}{N} g_N (1- s_N (\eta -h) + o(s_N))$ as long as
$\frac{i}{N} < h$. 
Consequently, in phase 1 we have
$\frac{\lambda_i}{\mu_i + \lambda_i} \geq\frac{1 + s_N(\eta -h) + o(s_N)}{2}$ and 
$\frac{\mu_i}{\mu_i + \lambda_i} \leq\frac{1 - s_N(\eta -h) + o(s_N)}{2}$. 

As long as $\frac{N-i}{N}\geq 1-h$ we can couple $\xi'$ with a continuous time binary Galton-Watson process $\mathcal{W}$ with individual birth rate
$g_N (1 + s_N(\eta-h) + o(s_N))(1-h)$ and individual death rate $ g_N(1 - s_N(\eta-h)+ o(s_N))(1-h)$, such that $T^{\xi'}_h \leq T_h^\mathcal{W}$ 
almost surely.

The probability that a single line in $\mathcal{W}$ will not be immortal is given by
that solution of the equation
$\frac{1-s_N(\eta-h) + o(s_N)}{2} + \frac{1+ s_N(\eta-h) + o(s_N)}{2} z^2= z$, which is
smaller than 1, see (Athreya and Ney, 1972, Chapter I.5).

Thus,  whenever an immortal line splits, the new line has a
chance $1-\frac{1-s_N(\eta-h) + o(s_N)}{1+ s_N(\eta-h) + o(s_N)} \leq  2s_N(\eta-h) + o(s_N)$ to be immortal.
Therefore we can couple $\mathcal{W}$ conditioned to hit $h$ with a pure-birth process $\mathcal{G}$ with birth rate 
$ 2 s_N g_N(\eta-h)(1-h) + o(s_N),$  such that
$T^\mathcal{W}_h \leq T_h^\mathcal{G}$ almost surely.

Because of
\[\mathbbm{E}[T_{h}^{\mathcal{G}}] =
\sum_{i=1}^{h N-1} \frac{1}{2 i(1-h) g_N s_N(\eta-h)} = \frac{ \log(h N)}{(s_N + o(s_N)) g_N} + \mathcal{O}\left(\frac{1}{s_N g_N} \right)\]
\[\mathbbm{V}[T_{h}^{\mathcal{G}}]= \sum_{i=1}^{hN -1} \frac{1}{4(1-h)^2  (s_N + o(s_N) )^2(\eta-h)^2 i^2 g_N^2} =  \mathcal{O}\left(\frac{1}{(g_Ns_N)^2} \right)\] 
we can estimate 
\begin{align*}
\mathbbm{P}\left(T^{\xi^N}_h > \frac{1}{s_N^{1+\epsilon} g_N}\, \big |\, T^{\xi^N}_h <\infty \right) &=  \mathbbm{P}\left(T^{\xi'}_h > \frac{1}{s_N^{1+\epsilon} g_N} | T^{\xi'}_h <\infty \right) \\
& \leq \mathbbm{P}\left(T_h^{\mathcal{G}} > \frac{1}{s_N^{1+\epsilon} g_N} \right) \leq \frac{1/s_N^2}{1/s_N^{2(1+\epsilon)}} \rightarrow 0\,.   
\end{align*}

In phase 2 we rescale time in the process $\xi^N$ by a factor $1/(g_N s_N)$ and denote the corresponding process by $\mathcal{V}'$. Then the infinitesimal mean of the time
rescaled process $\mathcal{V}'$ 
 equals $2 V'(1-V')(\eta - V')$  in state $V'$ and the infinitesimal variance equals 
$ \frac{1}{s_N N} V' (1-V') + o(\frac{1}{N})$. Since $s_N N \rightarrow \infty$, also $\frac{1}{s_N N} V' (1-V') \rightarrow 0$.
Consequently, by the dynamical law of large numbers (\cite{Kurtz1971}), $\mathcal{V}'$ converges to the solution of the differential equation 
$\dot{x}= 2 x(1-x)(\eta-x)$ with initial condition $x(0)= h$. The level $\eta-h$ is reached after
time $\mathcal{O}(1)$ and consequently, when 
rescaling time back the second phase can be estimated by $\mathcal{O}( N^{b }/g_N)$ time units whp.

In phase 3 we refine the argument of phase 2. Instead of rescaling time by $\frac{1}{s_N g_N}$ we rescale by a larger
factor $k_N/g_N$, so that $N/k_N$ still converges to $\infty$. In this manner the time to reach a level $\eta - h_N$
can be estimated by $\mathcal{O}(k_N^{1+\delta}/g_N)$ whp, if $k_N h_N s_N \xrightarrow{N\rightarrow \infty} 1$ for any $\delta>0$.

In order to show \eqref{firsteq} we choose $h_N = s_N^{a +\epsilon_1} $ and consequently arrive at $k_N = N^{b(a +\epsilon_1 +1)}$. By choosing an appropriate $\delta$ the claim follows.

For showing \eqref{secondeq} we choose $k_N = N^{(1+2b)/3}$ and $h_N = N^{-(1-b)/3}$.

In the last phase 4 we set $s_N=0$ and $r'_{N,1}= r'_{N,2}=0$. This gives an upper bound on the time $T_4 = T^{\xi^N}_\eta - T^{\xi^N}_{\eta -h_N}$ to
reach the level $\frac{\lceil \eta N \rceil}{N} $ from $\frac{\lceil (\eta-h_N)N \rceil}{N}$. 
Rescale time by $N/(\eta(1-\eta) g_N)$. For~$i$ close to $ N \eta$ we can
estimate $i(N-i)/(\eta(1-\eta))$  by $ N^2$ and hence by ignoring balancing selection we arrive at a random walk
with increments $\pm\frac{1}{N}$ occurring at rate $N^2$.

Until  the level $\frac{\lceil \eta N \rceil}{N}$ is hit the process may return to the level $\frac{\lceil ( \eta-h_N)N \rceil}{N}$
several times. Since below level $\frac{\lceil ( \eta -h_N)N \rceil}{N}$ the approximation of phase 3 holds, we can 
approximate the path by the excursions of a rescaled random walk to the levels $\frac{\lceil \eta N\rceil}{N}$ and $\frac{\lceil (\eta-2h_N)N \rceil}{N}$. If the random walk hits
the level $\frac{\lceil (\eta- 2 h_N)N \rceil}{N}$ it returns to the level $\frac{\lceil( \eta-h_N)N \rceil}{N}$ within a time of order $\mathcal{O}(k_N^{1+\epsilon})$ (according to phase 3).
Hence, we can estimate the time $T_4$ by
$T_4 \leq  (\sum_{i=0}^{R} ( \frac{N K_i}{(\eta(1-\eta)} + L_i) + S)/g_N,$
where $R$ is geometrically distributed with parameter $1/2$ (it counts the number of hits of the level $\frac{\lceil( \eta - 2 h_N)N\rceil}{N}$ before
the level $\frac{\lceil \eta N \rceil}{N}  $ is hit) and the
$K_i $ are independent copies of  $ T_{h_N} $ specified in Lemma \ref{rw}. The random variables $L_i$ are the lengths of the transitions starting from $\frac{\lceil( \eta - 2 h_N)N \rceil}{N}$ to $\frac{\lceil (\eta - h_N)N \rceil}{N}$ and
finally $S$ is the length of the last transition
from $\frac{\lceil( \eta - h_N)N \rceil}{N}$ to $\frac{\lceil \eta N \rceil}{N}$ (which does not hit the level $\frac{\lceil (\eta - 2 h_N)N \rceil}{N}$).

Choosing $h_N = N^{-(1-b)/3}$ yields $k_N = N h_N^2$ and hence with Lemma \ref{rw} we can estimate
\[\mathbbm{E}[T_4] \leq (\mathbbm{E}[R] \mathbbm{E}[{\color{black}c_1} N K_1 + L_1]  + \mathbbm{E}[S])/g_N \leq c_2( h^2_N N + \mathcal{O}(k_N^{1+\delta}))/g_N\]
for any $\delta>0$ and appropriate constants $c_1, c_2>0$ which are needed due to rescaling of time and by adding the different summands.
Furthermore,
\begin{align*}
\mathbbm{V}[g_N T_4] & \leq \mathbbm{E}[ (\sum_{i=0}^{R} (N K_i + L_i) + S)^2] % \\ & \leq \mathbbm{E}[(\sum_{i=0}^{R}(K_i +L_i))^2 + S \sum_{i=0}^{R} K_i + L_i  + S^2] 
\\ & \leq \sum_{j=0}^{\infty} \mathbbm{P}(R=j)\mathbbm{E}[ (\sum_{i=0}^{j} (N K_i + L_i) + S)^2]
 \\ &\leq (\mathcal{O}(k_N^{2+\delta}) + \mathcal{O}(h_N^4 N^2)) \sum_{j=0}^{\infty} \frac{1}{2}^{j+1} (j^2 + 6j +1) . 
\end{align*}
and since $h_N^4 N^2 = \mathcal{O}(k_N^{2+\delta})$ we obtain
\[\mathbbm{P}(T_4 < k_N^{1+\delta}/g_N)  \xrightarrow{N\rightarrow \infty} 0.\]
By choosing $\delta$ appropriately \eqref{secondeq} follows. 

\end{proof}

\begin{lem}[Length of non-effective excursions]\label{lnoeff}
Let Assumptions $(\mathcal{A})$ be fulfilled and let $\xi^N$ be the same process
as in Lemma \ref{balprob}. Assume $\xi^N(0) =1/N$. 
 Let $\tau^N_0= \inf \{t >0 | \xi^N(t) =0\}$ and let $\tau^N_\eta = \inf \{t >0 | \xi^N(t) =\frac{\lfloor \eta N \rfloor}{N}\}$. 
 Then 
 \begin{align}\label{lenI}
 \lim_{N\rightarrow \infty} \mathbbm{P}(\tau^N_0 < N^{b + \epsilon}/g_N | \tau^N_0 < \tau^N_\eta) =1 
 \end{align}
 for any $\epsilon >0.$
\end{lem}

\begin{proof}
Because of Remark \ref{heightbound} we may replace the event $\{\tau_0^N < \tau_\eta^N\}$ in \eqref{lenI}
by $\{\tau_0^N < \tau_{\delta_N}^N\}$ for a sequence $\delta_N$ converging to 0 sufficiently slowly.

In the following we rescale time by $1/g_N$ and space by $N$ and denote the corresponding process by $\xi'$, i.e. $\xi'(t) = N\xi^N(t/g_N)$.
Let $\mathcal{W}:= \mathcal{W}^N= (W^N_s)_{s\geq0}$ be a linear birth-death process starting in $W^N_0=1$
 with individual birth rate 
$(1-\delta_N) (1+ s_N \eta + o( s_N) )$ and individual death rate $(1-\delta_N)(1-s_N (\eta- \delta_N) + o( s_N))$.
We can couple the process $\xi'$ with $\mathcal{W}$ as long as $\xi' < \delta_N N$, such
that the hitting time of 0 of $\mathcal{W}$ is stochastically
larger than that of $\xi'$ and that \[ \lim_{N\rightarrow \infty} \frac{\mathbbm{P}( \tau^W_0 < \tau^W_{\delta_N}) }{\mathbbm{P}( \tau^N_0 < \tau^N_{\delta_N} ) }  =1 \]
with $\tau_x^W= \inf\{t >0| W^N_t = \lfloor x N \rfloor\}$.

In order to prove \eqref{lenI} it thus suffices to show
$\lim_{N\rightarrow \infty} \mathbbm{P}(\tau^W_0 < N^{b + \epsilon} | \tau^W_0 < \tau^W_{\delta_N}) =1 $.

We can interpret the process $\mathcal{W}$ also as an time-continuous binary branching process with individual reproduction
rate $(1-\delta_N)( 2 -  s_N \delta_N + o(s_N))$ and with offspring distribution weights $p_0 = \frac{1-s_N(\eta -\delta_N) + o(s_N)}{2 + s_N \delta_N + o(s_N)}$ and 
$p_2 = \frac{1+ s_N \eta + o(s_N)}{2 + s_N \delta_N + o(s_N)}$.

Denote by $E$ the event $\text{``$\mathcal{W}$ goes extinct''}$.
It suffices to show 
\begin{align}\label{hb2}
\lim_{N\rightarrow \infty}  \mathbbm{P}(\tau^{W}_0 < N^{b + \epsilon} | E) =1, 
\end{align}
because then we can argue as follows.

We have
$ \mathbbm{P}(E) =  \mathbbm{P}(E, \tau^W_0 < \tau^W_{\delta_N}) +  \mathbbm{P}(E, \tau^W_0  > \tau^W_{\delta_N}) $
and
\[ \mathbbm{P}(E, \tau^W_0 > \tau^W_{\delta_N}) \leq  \mathbbm{P}(E, W^N_0= \lfloor {\delta_N} N \rfloor) \leq (1- 2s_N \eta + o(s_N))^{\eta N} = \mathcal O(\exp(-N^{1-b})).\]
Consequently, 
\begin{align*}
 \mathbbm{P}(\tau^W_0 < N^{ b +\epsilon} | E) & = \frac{ \mathbbm{P}(\tau^W_0 < N^{b +\epsilon} , E)}{ \mathbbm{P}(E)} \\ 
& \leq \frac{  \mathbbm{P}(\tau^W_0 < N^{b + \epsilon} \cap E)}{ \mathbbm{P}(\tau^W_0 < \tau^W_{\delta_N}) +  \mathcal O(\exp(-N^{1-b}))} \\
& \leq  \frac{\mathbbm{P} (\tau^W_0 < N^{b +\epsilon } \cap \tau^W_0 < \tau^W_{\delta_N})}{  \mathbbm{P}(\tau^W_0  < \tau^W_{\delta_N}) + \mathcal O(\exp(-N^{1-b}))} \\
& = \frac{\mathbbm{P}(\tau^W_0 < N^{ b + \epsilon} \cap \tau^W_0 < \tau^W_{\delta_N})}{\mathbbm{P}(\tau^W_0 < \tau^W_{\delta_N})
(1 + \frac{ \mathcal O(\exp(-N^{1-b}))}{ \mathbbm{P}(\tau^W_0 < \tau^W_{\delta_N})})} \\
& = \mathbbm{P}(\tau^W_0 < N^{b + \epsilon} | \tau^W_0  < \tau^W_{\delta_N}) (1 + o(1)), 
\end{align*}
since $\mathbbm{P}( \tau^W_0  < \tau^W_{\delta_N}) = 2 \eta s_N + o(s_N)  .$

To show \eqref{hb2} recall that as in the discrete case the offspring distribution of the branching process conditioned on extinction 
has the weights $\frac{ p_0}{q}$ and $ p_2 q$, where $q = 1- 2 s_N \eta + o(s_N)$ denotes the probability of extinction.

Thus the generating function of the conditioned process, $F(s,t)= \mathbbm{E}[s^{W^N_t}| E]$, 
solves the differential equation
\begin{align*}
\frac{\partial F(s,t)}{\partial t} = & \frac{ (1-\delta_N) (1-s_N (\eta -\delta_N) + o(s_N) )}{q}  \\  & \quad \quad
- (1 -\delta_N )\left( \frac{(1-s_N(\eta -\delta_N) + o(s_N))}{q} +  (1+ s_N \eta + o(s_N)) q\right) F(s,t) \\  & \quad \quad  + (1-\delta_N)(1 + s_N \eta + o(s_N)) q F(s,t)^2
\end{align*}
with initial condition $F(s,0)=s$.

The conditional expectation $M_t:= \mathbbm{E}[W^N_t| E] = \frac{\partial F(0,t)}{\partial s}$ solves
\[\frac{d M_t}{d t}= (1 -\delta_N )((1+ s_N \eta + o(s_N))q -  \frac{(1-s_N(\eta -\delta_N) + o(s_N))}{q}) M_t\]
with $M_0=1$, hence $M_t=\exp\left( (1 -\delta_N )((1+ s_N \eta + o(s_N))q -  \frac{(1-s_N(\eta -\delta_N) + o(s_N))}{q}) t\right)$.

Since $q= 1- 2 s_N \eta + o(s_N)$ we arrive at 
$M_t = \exp\left(- (2 s_N \eta + o(s_N))t\right).$ 

Consequently, for any $\epsilon >0$
\[ \mathbbm{P}(\tau^W_0 > N^{ b + \epsilon} | E) = \mathbbm{P}(W_{N^{b+ \epsilon}} >0 |E  )
\leq \mathbbm{E}[ W_{N^{b+\epsilon}} | E] \leq \exp\left(-(2 s_N \eta + o(s_N))N^{b+\epsilon}\right),\]
which yields \eqref{hb2}
\end{proof}

\begin{lem}[Time to leave balance]\label{stab}
Let Assumptions $(\mathcal{A}1)$, $(\mathcal{A}2)$ and $(\mathcal{A}3')$ be fulfilled.
Consider a Markov process $\xi^N= (\xi^N_t)_{t\geq0}$ on $\{0,1/N, ...,1 \}$ with the same transition rates as in Lemma~\ref{balprob}. 

Start $\xi^N_0$ in $\frac{\lceil ( \eta \pm s_N^{a +\epsilon_1})N \rceil}{N}$  
for some $\epsilon_1 < \epsilon$
and let
\[\tau_{U} =\inf \{t \geq 0 | \xi_t= \frac{ \lfloor N( \eta \pm s_N^a) \rfloor }{N} \}.\]

Then whp

\[\tau_{U}^{-1} = \mathcal{O}( g_N \exp(-N^{1 - b (2a+1 + \epsilon)})).\]
\end{lem}

\begin{proof}
We consider as in the proof of Lemma \ref{balprob} the process $\xi' = N \xi^N$.

It suffices to estimate the time required to reach the level $q_2=\lfloor ( \eta + s_N^a)N \rfloor$ from $q_1=\lceil N( \eta+ s_N^{a +\epsilon_1}) \rceil $.
Consider first
the probability $p_{d}$ to reach from $a$  the level  $ a-1 $ 
before $b$.
We estimate this probability from below by applying Lemma \ref{ruin} with  $N_1:= 1, N_2:=   N( s_N^a - s_N^{a +\epsilon_1} )$ and $p:= \frac{1}{2}(1- s_N^{1 +a +\epsilon_1} + o(s_N^{1 +a +\epsilon_1}))$. Hence \begin{align*}
p_d & \geq  \frac{ \left( \frac{1 -s_N^{1 +a +\epsilon_1} + o(s_N^{1 +a +\epsilon_1}) }{ 1 + s_N^{1 +a +\epsilon_1} +  
o(s_N^{1 +a +\epsilon_1})} \right)^{N( s_N^a- s_N^{a +\epsilon_1} )}  -1}
{ \left( \frac{1 -s_N^{1 +a +\epsilon_1} + o(s_N^{1 +a +\epsilon_1}) }{ 1 + s_N^{1 +a +\epsilon_1} +  
o(s_N^{1 +a +\epsilon_1})} \right)^{N( s_N^a- s_N^{a +\epsilon_1} ) +1}  -1}   
\\ &% = \frac{ \exp(-N^{1- 2.5 b})-1}{\exp(-N^{(1- 2.5 b)}) -1} + o(\frac{ \exp(-N^{(1-b)(1-x-y)}) -1}
%{\exp(-N^{(1-b)(1-x- y)}) -1}) \\
 \geq 1 - 2 \exp(-N^{1- b(2a +1 +\epsilon_1)}) +o(\exp(-N^{1 -b(2a +1 +\epsilon_1)}))
\end{align*}

Now couple $\xi'$ with a Markov process $\mathcal{W}=(W_t)$ with state space $\{q_1-1, q_1, q_2\}$, such that
$\tau^{\xi'^N}_{s_N} < \tau^{\mathcal{W}}_{s_N} = \inf \{t \geq 0 |W_t = q_2\}$. 
To this purpose we define the jump rates of $\mathcal{W}$ as follows: From $q_1$, $\mathcal{W}$ jumps  at rate $g_N$ to state $q_1-1$ with probability $p_d$  and to $q_2$ with probability $1- p_d$.
From state $q_1 -1$ the process $\mathcal{W}$ returns instantly to state $q_1$.
Then the number of trials needed to enter state $q_2$ when starting in state $q_1$ is geometrically distributed with
mean $\frac{1}{1-p_d},$ which gives the desired estimate.

\end{proof}

\begin{lem}[Typical host states as $N\to \infty$]\label{aspure} Let $\overline t > 0$. The following holds in the situation of Theorem \ref{TfiniteM} for each  $i \in \{1,\ldots, M\}$, and in the situation of Theorem \ref{MF} for 
each $i \in \{1,\ldots, k\}$: At any reinfection event and at any host replacement event  in which
 host $i$ is involved in the time interval $(0,\overline t]$, this host is 
whp in a state in $\{0,1\}\cup U^{\eta,N}$.
\end{lem}

\begin{proof}
First of all note that the total rate of reinfection events at which host $i$ is infecting another host or host $i$ is infected by another host, is 
$2r/s_N$ and the total rate of host replacement events is 2.

\begin{itemize}
\item (whp no reinfections and no host replacements occur during non-effective excursions)
 The length of an non-effective excursion can be bounded from above by $N^{b + \epsilon}/g_N$  
 with probability $1-\exp(-{c_1 N^b})$ for an appropriate constant $c_1>0$, see the proof of Lemma \ref{lnoeff}. 
 The probability that neither another reinfection event nor a host replacement event happens during a non-effective excursion can be estimated from below by
\begin{align*} 
&   (1-\exp({-c_1 N^b})) 
( \exp({- (\frac{2r}{s_N} +2) \frac{N^{b+\epsilon}}{g_N}})) \\ & \sim  (1-\exp(-{c_1 N^b}))(1-  \frac{ N^{2b + \epsilon}}{g_N}),
\end{align*}
since $g_N \gg N^{3b +\epsilon}$. With $c_2:=2\overline t$, the probability $p_N$ of the event that  the number
of reinfection events at which host $i$ is reinfected within the time interval $[0,\overline t]$
obeys $p_N \rightarrow 1$.
Hence, the probability that during none of the non-effective excursions occurring within the time interval $[0,t]$ a reinfection event happens
can be estimated from below by
\[p_N (1-\exp(-{c_1 N^b}))^{2 N^{b}/r} (1- \frac{N^{2b  + \epsilon}}{g_N})^{c_2 N^{b}/r} \rightarrow 1.\]

%Note, that Claim b) holds since as well 
%\]p_N (1- \frac{N^{\frac{7b}{2}  + \delta}}{g_N})^{2 N^{b}/r} \rightarrow 1 \]
%for $g_N \gg N^{5b}$.

\item (whp no reinfections and no host replacements during transitions from a boundary state to $D^{\eta,N}$)
Consider the event $E_N$ that the duration of a transition
from state 0 or 1 to $D^{\eta,N}$ is smaller than $N^{b(1+a) +\epsilon}/g_N$. 
By Proposition~\ref{TimeToEta} there is a sequence $\delta_N \rightarrow 0$ such that
\[\mathbbm{P}(E_N) \geq 1 -\delta_N.\] 
Since the probability that a reinfection event is effective
 is proportional to $s_N$, the probability $q_N$ that the number of effective reinfection events is 
 larger than $k_N$ converges to 0 as $N\to \infty$, where $k_N$ is an arbitrary sequence with $k_N\rightarrow \infty$. 
 
The probability that within a time interval of length $ N^{b(1+a) +\epsilon}/g_N$ a reinfection event or a host replacement event happens in which host $i$ is involved 
can be estimated from above by
\[ 1 - \exp(-\frac{2 r N^{b(2+a) +\epsilon}}{g_N}) \sim \frac{2r N^{b (2+a) +\epsilon}}{g_N}.\]
Consequently, we estimate the probability that no reinfection and no host replacement event happens during a transition from 
the boundary to  $D^{\eta, N}$ by 
\[q_N (1-\delta_N)^{k_N} (1- \frac{N^{b (2+a) +\epsilon}}{g_N})^{k_N}.\]
This converges to 1 if we choose  $k_N \leq \min\{ \frac{1}{\sqrt{\delta_N}}, \frac{\sqrt{g_N}}{\sqrt{N^{b(2+a) +\epsilon}}}\}$. 
\item (remainig time) 
      If a host is initially in state 0 or 1, then apart from remaining in a pure state she (or the host that eventually replaces her) experiences only 
    non-effective excursions  until she is effectively reinfected. The same applies to the time immediately after a host replacement. If, however, the state of a host is initially in $D^{\eta,N}$
      or has reached  $D^{\eta,N}$ after an effective reinfection, then it remains within  $U^{\eta,N}$ whp until the next host replacement event.  Indeed, without
      host replacement, starting from a frequency in $D^{\eta,N}$, it follows
  from Lemma \ref{stab}  that the type $A$-frequency remains whp within the set $U^{\eta,N}$ for at least
  $\exp(N^{1 -b (2a +1 +\epsilon)})/g_N$ time units. Consequently, due to Assumption ($ \mathcal A 3)$, and because host replacements come at a positive rate,
the set $U^{\eta,N}$ is whp left only because of a host replacement. As the number of times a host reenters the interval $U^{\eta,N}$
      can whp be bounded by any sequence $k_N$ with $k_N\rightarrow \infty$, we conclude that whp in any phase between a successful reinfection 
      and a host replacement event the host's type $A$-frequency
      remains in the interval $U^{\eta,N}$.    
  \end{itemize}
\end{proof}

\begin{proof}[Proof of Theorem \ref{TfiniteM}]
Since $M$ is fixed, we will briefly write ${\bf X}^N$ instead of ${\bf X}^{N,M}$.
 Let $\alpha > 0$ and $0<\underline t< \overline t$ be as in the Theorem's assumption. As in the previous lemmata of this subsection, we choose a constant $a \in (0, \frac{1-b}{2b})$ such that $(\mathcal{A}1)$, $(\mathcal{A}2)$ and $(\mathcal{A}3')$ are fulfilled, cf. Remark \ref{rem2_3}.

To prove the convergence of ${\bf X}^N$ to $\bf Y^M$ we will construct a process $\hat{{\bf X}}^{N,M}$, also denoted by $\hat{{\bf X}}^N$ for short, in a similar graphical way as  $\bf Y^M$, and
then {\color{black}we} will couple ${\bf X}^N$ and $\hat{{\bf X}}^N$.

To this purpose, let  $\xi^N$ be the process defined in Lemma \ref{balprob} with $r'_{N,1}= r'_{N,2}=0$, and abbreviate its ``probabilities to balance'' by
\begin{equation}\label{probabal}
p^0_N :=\mathbbm{P}_{\frac{1}{N}}(\tau_{\lceil (\eta- s_N^{a +\epsilon_1})N\rceil} < \tau_0); \qquad p^1_N := \mathbbm{P}_{\frac{N-1}{N}}(\tau_
{\lfloor N(\eta + s_N^{a +\epsilon_1}) \rfloor} < \tau_N).
\end{equation}

Construct $\hat{{\bf X}}^N$ on the time interval $[0,\overline t]$ as follows:
\begin{itemize}
 \item[(i)] Between each pair $(i,j)$, with $i, j\in \{1, ..., M\}$ draw HR arrows at rate $1/M$ and PER arrows at rate $r_N s_N/M$.
 \item[(ii)] At time 0 initialize hosts 1, ..., $M$ according to the initial distribution of ${\bf X}^{N}$.
 \item[(iii)] Then let the intra-host Moran processes evolve within each host until a PER or HR arrow hits this host at some time $t$. 
       If at that time an HR arrow  is shot from host $i$ to host $j$, then toss a $0$-$1$-coin whose outcome is $1$ with probability $\hat X_i(t-)$, and 
       set $\hat X_j(t)$ equal to the result of the coin toss. 
       If at time $t$ a  PER arrow is shot from host $i$ to host $j$,  then check the state of host $j$ at time $t-$.  is in state $0$ or $1$. If it is neither $0$ nor $1$, then ignore the arrow. If it
   is either $0$ nor $1$, then toss a $0$-$1$-coin whose outcome is $1$ with probability $\hat X_i(t-)$. The outcome of this coin toss gives the type being
       transmitted at the reinfection event. If $\hat X_j(t-)$ coincides with the outcome of the coin toss, then ignore the arrow. If $\hat X_j(t-)$ does not
       coincide with the outcome of the coin toss, then toss
       a second $0$-$1$-coin 
       now with success probability $p^{\hat X_{j}(t-)}_N/s_N$. If the result of the second coin toss is 0, then ignore the arrow. If the result of the second coin
       toss is~1, then start
       in host $j$ a type $A$-frequency path
       according to the process~$\xi^N$ from Lemma~\ref{balprob} with $r'_{N,1}= r'_{N,2}=0$, starting in $1/N$ and conditioned to reach the state
       $\frac{\lceil(\eta- s_N^{a +\epsilon_1}) N\rceil}{N}$ if $\hat X_j(t-)=0$, and starting in $(N-1)/N$ and conditioned to reach the state 
       $\frac{\lfloor (\eta+ s_N^{a +\epsilon_1}) N\rfloor}{N}$ if $\hat X_j(t-)=1$. 
       Afterwards perform an (unconditioned) walk according to the process $\xi^N$ with $r'_{N,1}= r'_{N,2}=0$. 
\end{itemize}

       Next we show that 
       \begin{itemize} 
        \item[1)] the finite-dimensional distributions of $\hat{{\bf X}}^N=\hat{{\bf X}}^{N,M}$ converge to those of $\bf Y^M$ as $N\to \infty$,
      \item[2)]  ${\bf X}^{N}$ and $\hat{{\bf X}}^{N}$ can be coupled such that for all $\tilde{\delta} >0$ and $t\in [0,\overline t]$
       \begin{align}\label{coupleXhatX}
       \lim_{N\rightarrow \infty} \mathbbm{P}(\max_{i=1, ..., M} |X_i^N(t) - \hat{ X}_i^N(t)| >\tilde{\delta})=0, 
       \end{align}
        (which implies that the finite-dimensional distributions of ${\bf X}^N$ and $\hat{{\bf X}}^N$ have the same limiting distributions),
        \item[3)] the sequence ${\bf X}^N$ is tight with respect to the Skorokhod $\mathrm M_1$-topology on the time interval~$[\underline t,\overline t]$. 
       \end{itemize}
       
       Proofs of claims 1)-3):
       
       1) If $\hat{X}^N_i(0) \in [\alpha, 1-\alpha]$, then we know from Proposition \ref{TimeToEta} that $\hat{X}^N_i$ reaches 
       whp the interval $D^{\eta,N}$ (as defined in \eqref{Deta}) before host~$i$ is affected by a reinfection or a host replacement; recall that $1/ g_N =o r_NN^{-b\min\{1+a,2\} -\epsilon}$ by Assumptions $(\mathcal A1)$,  $(\mathcal A2)$ and $(\mathcal A3) (i)$.
       Each component of~$\hat{{\bf X}}^N$ starting in  state $0$ or $1$ remains in that state until an effective reinfection event happens, which then
       results in a transition to
       $D^{\eta,N}.$  Because of Proposition \ref{TimeToEta} and because of the Assumptions $(\mathcal A1)$ and  $(\mathcal A2)$ the duration of the transition converges to 0 
       on the host time scale as $N\to \infty$. Furthermore the intervals $D^{\eta,N}$ shrink to $\{\eta\}$ as $N\to \infty$, and hence in that
       limit an effective reinfection event leads to a jump from $0$ or $1$ to $\eta.$
%       After an effective reinfection event the set $U^{\eta,N}$ can be left because of random fluctuations or due to a host replacement event. 
       According to Lemma \ref{stab} whp the exit from $U^{\eta,N}=U^{\eta,N}_a$ (as defined in \eqref{Ueta}) is caused by an HR event (and not by random fluctuations). 
    The probabilities $p^0_N$ and $p^1_N$ defined in \eqref{probabal} obey $p^0_N s_N \rightarrow 2 r \eta$ and $p^1_N s_N \rightarrow 2 r (1-\eta)$ as $N\to \infty$, hence we conclude that the rates at which transitions from state $0$ or $1$ 
      to $D^{\eta,N}$ occur converge
       as $N\to \infty$  to the jump rates of the process $Y^M$ from states $0$ or $1$  to $\eta$.
       
       Altogether we observe that the limiting process of $\hat{{\bf X}}^N$ is concentrated on $\{0, \eta,1\}$ and the rates at which transitions occur between the 
       states converge to those of $Y^M$. This proves 1). 
         
       2) In order to obtain the desired coupling of ${\bf X}^N$ and $\hat{{\bf X}}^N$, we condition on the event $E$ that in any host excursions from the states 0 and 1 (caused by reinfection) are non-overlapping and that any host whose state is performing a transition
        from state 0, resp. state 1, 
       to a state in $D^{\eta,N}$ is not affected by any further reinfection during this period. 
       By Lemma \ref{aspure}, $E$ is an event of high probability; hence, in order to check  \eqref{coupleXhatX}, we may tacitly condition on $E.$
     
      The probability   that in $\mathbf X^N$ a reinfection hitting a host in state 0 and coming from a host in state 1  becomes effective then turns into $p^0_N$. If a host
       is in state $0$ and is hit by a reinfection
       arrow that is shot by a host whose state is in $U^{\eta,N}$, then the probability at which the reinfection becomes
       effective is bounded from below by $(\eta-s_N^a) p^0_N$ and from above by $(\eta+s_N^a) p^0_N$. (A similar reasoning applies for
       hosts that are in state~$1$ when suffering a reinfection.) 
       Furthermore, on the event $E$, the distributions of a transition path from 0 to $\eta -s_N^{a +\epsilon_1}$
       and from 1 to $\eta+ s_N^{a +\epsilon_1}$, resp., in ${\bf X}^N$ and $\hat{\bf X}^N$ coincide.
       Whenever host $i$ is effectively reinfected it performs whp in ${\bf X}^N$ as well as in $\hat{\bf X}^N$
       random fluctuations within $U^{\eta,N}$ until a HR event turns the state of host $i$ into 0 or 1 again according to Lemma \ref{stab}.
       To couple ${\bf X}^N$ and $\hat{\bf X}^N$ we can assume that these random fluctuations are performed in ${\bf X}^N$ and $\hat{\bf X}^N$
       independently.  At the times $t$ of host replacements, independent uniform random variables $U_t$ are drawn from $[0,1]$.  If $U_t \leq X^N_i(t-)$, then the
       replaced host in $\mathbf X_N$ jumps to state $1$. Similarly, If $U_t \leq \hat{X}^N_i(t-)$, then the replaced host in $\mathbf X_N$ jumps to state $0$.
 Since $|X^N_i(t-) - \hat{X}^N_i(t-)| \leq 2s_N^a$ whp, the same type is transmitted to the replaced host
       whp.
       
       Due to Assumptions $(\mathcal A)$ and because the total number $M$ of hosts does not depend on $N$, the total number of effective reinfection events as well
       as the total number of host replacement events in $[0,\overline t]$ is with high probability bounded by $k_N$ for any sequence
       $k_N$ that tends to $\infty$ as $N\rightarrow \infty$. Thus, for each $s \in [0,t]$  it follows that whp the distance of each pair of
       components $(X_i^N(s), \hat{X}_i^N(s))$ of $ {\bf X}^N(s)$ and $ \hat{{\bf X}}^N(s)$ 
  has distance $\le 2s_N^a$. This implies  \eqref{coupleXhatX}.

3) It remains to prove the tightness of the sequence $(\mathbf X^N)_{N\ge 1}$
%= ((X^N_1, ..., X_M^N))_{N\ge 1}$ 
in the  \mbox{$\mathrm M_1$-topology} on the interval 
$[\underline t, \overline t]$.   By \cite{Skorokhod1956}, Theorem 3.2.1, 
it suffices to show that for all $\delta_1 >0$
\begin{equation}\label{modcont}
\lim_{\delta_2\to 0} \limsup_{N \to \infty} \max_{i=1,\ldots, M}\mathbbm{P}(\omega(X_i^{N}, \delta_2) > \delta_1) =0,
\end{equation}
where the $\mathrm M_1$-modulus of continuity of the path $X_i^{N}$ with resolution $\delta$ is defined as
\begin{equation}\label{omegaM1}
\omega(X_i^{N}, \delta):= \sup_{\underline t\le t\le \overline t}\omega_t(X_i^{N}, \delta)\end{equation}
with 
\begin{equation}\label{omegaM1t}
\omega_t(X_i^{N}, \delta):= \sup_{\underline  t\le t'\le t\le t''\le \overline t\,:\,  t''-t'\le \delta}d(X_i^M(t), [X_i^M(t')\wedge X_i^M(t''), X_i^M(t')\vee X_i^M(t'')])
\end{equation} 
and $d$ denoting the Euclidian distance in $\mathbb R^M$.

a) We first observe that for sufficiently small $\delta_2$ the probability that in the time interval $[0,\overline t]$  there are two host replacement or effective reinfection events that affect host $i$ and are in a temporal distance closer than $2\delta_N$ tends to $0$ as $N\to \infty$. 

b) We now discuss the impact of ineffective reinfections (resulting in excursions from the boundary points $0$ and $1$) and of effective reinfections (resulting in transitions from the boundary points to $U^{\eta, N}$) on the $\mathcal M_1$-modulus of continuity of $X_i^N$.  We obtain from the estimate \eqref{leveldelta} in the proof of Lemma \ref {balprob} that the probability for ${X}_i^{N}$, once it has reached a level $\delta_1$,  to return to~$0$ before it hits $\eta$, is exponentially small in $N^{1-b}$. Since the number of reinfections which a host suffers in the time interval $[0,t]$ is of the order of $r_N=O(N^b)$, for any $\delta_1>0$ the probability that there exists an excursion of the  type $A$-frequency in host $i$ that raises from $0$ above $\delta_1$ without coming close to $\eta$ before a host replacement, tends to $0$ as $N\to \infty$. Since we are interested in events of high probability, we may thus restrict our considerations to paths $X_i^N$ whose excursions from $0$ consist only of excursions staying below $\delta_1$, plus transitions that go from $\delta_1$ to $U^{\eta,N}$, without a visit to $0$ in between.  Then, if $X_i^N(t) < \delta_1$,  we clearly have $\omega_t(X_i^{N}, \delta_2) < \delta_1$. If  $\delta_1< X_i^N(t) < \eta-\delta_1 $, then we resort to the considerations concerning phase 2 in the proof of Proposition \ref{TimeToEta}. These  tell us that an excursion (on its way towards the level $\eta$) converges, when time is stretched by the factor $g_N s_N$, uniformly in probability  to a deterministic path which is monotone increasing. Consequently, for time points $t$ for which the excursion is above  $\delta_1$ but below $\eta-\delta_1$, $\omega_t(X_i^{N}, \delta_2)$ converges in probability to $0$ as $N\to \infty$,  uniformly in $t$. 
 
 c) Next, we discuss the impact of fluctuations around $\eta$ on \eqref{omegaM1}. According to Lemma~\ref{stab}  the corresponding frequency remains with high probability in the neighbourhood $U^{\eta,N}$, whose size shrinks to $0$ as $N\to \infty$, until the next host replacement kicks in. Since $M$ is finite and the number of transitions to $D^{\eta,N}$ is almost surely finite
 on the time interval $[\underline{t}, \overline{t}]$ whp all type frequencies remain within $U^{\eta,N}$ until host replacement events relocate type frequencies to the boundaries.
 
d) Finally we turn to time points  at which hosts are replaced. For such times $t$, and for large~$N$, all the states $X_i^N(t-)$, $i=1,\ldots, M$ are with high probability in $\{0,1\}\cup U^{\eta,N}$, see Lemma \ref{aspure}. Assume that  $X_i(t-)=0$ . Then the considerations in step b) of the proof tell us that $X_i^N(t')<\delta_1$ whp for all $t' \in [t-\delta_2, t)$, which implies $\omega_t(X_i^{N}, \delta_2) < \delta_1$ whp, irrespective of whether and $X_i(t)=1$ or  $X_i(t)=0$. An analogous statement holds for 
 $X_i(t-)=1$. Finally, for  the case  $X_i(t-)\in U{\eta,N}$   the argument in step c) shows that $\omega_t(X_i^{N}, \delta_2) < \delta_1$ whp, irrespective of whether $X_i(t)=1$ or  $X_i(t)=0$.

In summary, we obtain \eqref {modcont}.
\end{proof}

\begin{proof}[Proof of Theorem \ref{MF}]
We argue similarly as in the proof of Theorem \ref{TfiniteM}. Let $\hat {\bf X}^N:=  \hat{{\bf X}}^{N,M_N}$ be  as in the proof of Theorem \ref{TfiniteM}. 
We claim that for all $k \in \mathbb N$ and $0<\underline t <\overline t$
 \begin{itemize} 
        \item[1)] the finite dimensional distributions of $(\hat{X}^{N}_1, ...,\hat{X}^{N}_k) $ converge to those of $(V_1, ..., V_k)$,
      \item[2)] we can couple ${\bf X}^{N}$ and $\hat{{\bf X}}^{N}$ such that for all $\delta_1 >0$ and $t\in [0,\overline t]$
       \begin{align}\label{couple}
       \lim_{N\rightarrow \infty} \mathbbm{P}(\max_{i=1, ..., k} |{ X}_i^N(t) - \hat{ X}_i^N(t)| >{\color{black}\delta_1})=0, 
       \end{align}
       (which implies that finite dimensional distributions of $(X_1^N, ..., X_k^N)$ and $(\hat{ X}_1^N, ..., \hat{X}^N_k)$ have the same limiting distributions),
        \item[3)] the sequence $(X_1^N, ..., X_k^N)$ is tight with respect to the Skorokhod $\mathrm M_1$-topology on $[\underline t,\overline t]$.
       \end{itemize}

The proof of these claims   
follows along the same lines as the proof of Theorem \ref{TfiniteM}. Again we construct
the HR and PER-arrows between pairs of hosts, where the latter arise as a thinning (with retainment probability $s_N$) of the reinfection arrows.
However, now we keep track only of those arrows which contribute to the ``history'' of hosts $1,\ldots, k$. More specifically, we follow back also the
lineages of those hosts that shot the HR and the PER-arrows which hit hosts $1,\ldots, k$ (call these shooters the ``primary'' hosts)
and the lineages of the ``secondary'' hosts
that shot the HR and PER-arrows which hit the ``primary'' hosts, etc.  An essential point is that the  number of all hosts that are involved in this ``influence graph'' (of hosts $1,\ldots,k$ along the time interval $[0,\overline t]$) remains tight as $N \to \infty$. Therefore the arguments from the proof of  Theorem \ref{TfiniteM} apply for proving the claim 2).

The other important point is that, due to the
assumption $M_N \to \infty$ as $N\to \infty$, with high probability all the replacement and potential effective reinfection arrows
that are involved in the history of hosts $1, ..., k$ back from time $\overline t$, are shot by pairwise different hosts. Consequently, the sequence of influence graphs of hosts $1,\ldots, k$ in the time interval $[0,\overline t]$ converges, as $N\to \infty$ to a forest of $k$ trees, which are i.i.d. copies of the tree $\mathcal T_{\overline t}$ specified in Definition \ref{graphdyn}. Claim 1) then follows because of Corollary \ref{treeV}.

 Finally, claim 3) is shown in complete analogy to the corresponding claim 3) in  Theorem \ref{TfiniteM}  \end{proof}

\begin{proof}[Proof of Corollary \ref{EmpDist}] 
a) This is  shown in the same way as Corollary \ref{det1}. The role which in that proof was played by  Proposition \ref{chaos1} is now played by Theorem \ref{MF}, and the  Skorokhod $\mathrm J_1$-topology is now replaced by the $\mathrm M_1$-topology.

b) This follows from part a) by projection to the time point $t$ together with Corollary \ref{treeV}.
\end{proof}

\subsection{Proof of Theorem \ref{diversity}}
We first derive thew following
 \begin{prop}\label{ergodic}
 Let Assumptions $(\mathcal A)$ and \eqref{stablecond} be valid and suppose $M_N\to \infty$ as $N\to \infty$. Let $\mu^N_0$ be the (possibly random) 
 empirical distribution of $\mathbf X^{N}(0) = \mathbf X^{N,M_N}(0)$ as defined in  \eqref{empdist}, and assume that  the components 
  $ X^{N,M_N}_1(0),...,$ $ X^{N,M_N}_{M_N}(0)$ are exchangeable. For some $\alpha > 0$ and  $\delta' >0$ assume that
  $\mu^N_0(\{0\}\cup [\alpha, 1-\alpha] \cup \{1\}) \to 1$ as $N\to \infty$, and that 
 \begin{equation}\label{diversebegin}
  \mu^N_0(\{0\} \cup [\alpha, 1-\alpha] ) \ge \delta'
  \mbox{ as well as }    \mu^N_0([\alpha, 1-\alpha] \cup \{1\} ) \ge \delta'
  \end{equation}
   whp as  $N\to \infty$. Then for each $\delta>0$ there exists a (sufficiently large)  $t>0$ such that
   \begin{equation}\label{neartou}
   \mathbb P((\mu^N_t(\{0\}),  \mu^N_t(U^{\eta,N}), \mu^N_t(\{1\})) \in W^{\delta,\mathbf u}) \to 1  \mbox{ as } N\to \infty,
   \end{equation}
   where $W^{\delta,\mathbf u}:= (u^0-\delta, u^0+\delta) \times (u^\eta-\delta, u^\eta+\delta) \times (u^1-\delta, u^1+\delta)$.
 \end{prop}
\begin{proof} a) Fix  $\delta, \delta' >0$. For $\varepsilon > 0$ we put $B^{\varepsilon,\mathbf u}:= \{(v^0,v^1) \in \myco: (v^0-u^0)^2+ (v^1-u^1)^2  <\varepsilon^2\}$,  where $\myco= \{(v^0,v^1): v^0+v^1\le 1\}$ is the triangle depicted in  Fig.~\ref{FigEq}. Because of the stability of $\mathbf u$ (see Proposition \ref{equilibria}), there exists an $\varepsilon>0$ such that for any initial condition
%for all sufficiently small $\varepsilon > 0$ and $(v^0, v^1)$ on the boundary of $B^{\varepsilon,\mathbf u}$ the vector $({\dot v}^0, {\dot v}^1)$ defined by the right hand side of \eqref{dynsys} points into the interior of $B^{\varepsilon,\mathbf u}$. We choose $\varepsilon$ so small that not only this property is satisfied but also $(v^0,v^\eta,v^1) \in W^{\delta/3,\mathbf u}$ for all $(v^0,v^1) \in B^{\varepsilon,\mathbf u}$. Hence 
 the solution $(\mathbf v_\tau)_{\tau\geq 0}$ of \eqref{dynsys}  remains in  $W^{\delta/4,\mathbf u}$ as soon as $(v^0_\tau, v^1_\tau)_{\tau\geq 0}$ has entered $B^{\varepsilon,\mathbf u}$.  Thus, because of the compactness of the set
%Using Proposition \ref{equilibria}  we choose $t$ so large that for each initial condition 
%$\mathbf v_0$ such that 
\begin{equation} \label{vdiverse}
 G_{\delta'} := \{(v^0,v^\eta, v^1)\in \Delta^3 : v^0_0+ v^\eta_0  \geq \delta' \mbox{ and } v^\eta_0 + v^1_0 \geq \delta'\}
\end{equation} 
there exists a $t>0$ such that $\mathbf v_t \in W^{\delta/4,\mathbf u}$, uniformly for all initial conditions $\mathbf v_0 \in G_{\delta'}$. 
%Indeed, the mapping $\mathbf v_0 \mapsto \inf\{\tau > 0: \mathbf v_{\tau'} \in W^{\delta/4,\mathbf u} \mbox{ for all }\tau' \ge \tau\}$ is continuous on $\Delta^3\setminus\{(1,0,0),(0,,0,1)\}$, and the set of all $\mathbf v_0$ thst satisfy \eqref{vdiverse} is a compact subset of $\Delta^3\setminus\{(1,0,0),(0,,0,1)\}$.

b) For $k = k_N$ denote by $\mu_t^{N,k}$ the empirical distribution of $(X^{N}_1(t),\ldots,X^{N}_k(t))$.

For $N$ large enough we can choose $k_N$ (as a sufficiently small fraction of $M_N$, but still large enough) such that the following holds: 

(i) whp, the $\mathbb R^4$-valued random variable $(\mu^{N,k}_t(\{0\}),  \mu^{N,k}_t(U^{\eta,N}), \mu^{N,k}_t((0,1)\setminus U^{\eta,N}), \mu^{N,k}_t(\{1\})$ is closer to $(\mu^N_t(\{0\}),  \mu^N_t(U^{\eta,N}),  \mu^N_t((0,1)\setminus U^{\eta,N}), \mu^N_t(\{1\})$ than $\delta/4$ in the max-norm,

(ii) whp, $\mu^{N,k}_t((0,1)\setminus U^{\eta,N})< \delta/4$,

(iii) whp, the (approximate) potential ancestries (back from time $t$ to time $0$)  of hosts $1, \ldots, k$ do not collide.

c) We define $\mathbf v_0^{(N)}$ as that random element of $\Delta^3$ whose components are the weights of the  image of the random
measure $\mu^N_0$ under the mapping $\phi: 0 \mapsto 0$, $(0,1) \ni x \mapsto \eta$, $1\mapsto 1$. 
Then, by the arguments in the proof of Theorem \ref{MF}, conditional on   $\mathbf v_0^{(N)}$ and the whp event of part b) iii), the $\mathbb R^k$-valued random variable $(\phi(X^{N}_1(t)),\ldots,\phi(X^{N}_k(t)))$ can be represented as  $(C_{t1}, \ldots, C_{tk})$, with  $k$ independent copies $C_{t1}, \ldots, C_{tk}$ of $C_t$, where $C$ is the process specified in Definition \ref{graphdyn} whose leaf-colouring is performed by $\mathbf v_0^{(N)}$.
Hence whp, conditional on  $\mathbf v_0^{(N)}$ the empirical distribution of $(C_{t1}, \ldots, C_{tk})$ coincides with that of $(\mu^{N,k}_t(\{0\}),  \mu^{N,k}_t((0,1)), \mu^{N,k}_t(\{1\})$.

Lemma \ref{repv} tells us that, again conditional on $\mathbf v_0^{(N)}$, the weights of the distribution  of $C_t$ equal the components of $\mathbf v_t^{(N)}$, where $(\mathbf v_\tau^{(N)})_{\tau\ge 0}$  is the solution of~\eqref{dynsys} with initial condition $\mathbf v_0^{(N)}$.
Thus, since whp  $\mathbf v_t^{(N)}$ is closer than $\delta/4$ to the empirical distribution of $(C_{t1}, \ldots, C_{tk})$, also $\mathbf v_t^{(N)}$  is whp closer to $(\mu^{N,k}_t(\{0\}),  \mu^{N,k}_t((0,1)), \mu^{N,k}_t(\{1\})$ than $\delta/4$ in the max-distance.
%, which is close to $(\mu^{N,k}_t(\{0\}),  \mu^{N,k}_t((0,1)), \mu^{N,k}_t(\{1\})$ because of the approximate coupling of  $(\phi(X^{N}_1(t)),\ldots,\phi(X^{N}_k(t)))$ and $(C_{t1}, \ldots, C_{tk})$. 

d) Because of part b) i) and ii), $(\mu^N_t(\{0\}),  \mu^N_t(U^{\eta,N}), \mu^N_t(\{1\})$ is whp closer than $\delta/2$ to $(\mu^{N,k}_t(\{0\}),  \mu^{N,k}_t((0,1)), \mu^{N,k}_t(\{1\})$ in the max-norm.  

In summary, by assumption \eqref{diversebegin},  $\mathbf v_0^{(N)}$ belongs whp to the set $G_{\delta'}$ defined in \eqref{vdiverse}, and hence, because of part a) of the proof, 
 $\mathbf v_t^{(N)}\in W^{\delta/4,\mathbf u}$ whp. 
Because of part c), $(\mu^{N,k}_t(\{0\}),  \mu^{N,k}_t((0,1)), \mu^{N,k}_t(\{1\})$ is whp closer to $\mathbf v_t^{(N)}$ than $\delta/4$, again in the max-norm. Finally by d) $(\mu^N_t(\{0\}),  \mu^N_t(U^{\eta,N}), \mu^N_t(\{1\})$ is closer to  $(\mu^{N,k}_t(\{0\}),  \mu^{N,k}_t((0,1)), \mu^{N,k}_t(\{1\})$ than $\delta/2$ in the max-norm. Altogether this implies the assertion \eqref{neartou} of the proposition.
\end{proof}
We are now prepared for the\\
\textit{ Proof of Theorem \ref{diversity}}. Let us remark right at the beginning that neither for the asymptotic estimates of the probabilities of local and global success of mutations (defined and carried out below) nor for the asymptotic estimates of the duration of the transition in question we need  take into account the influence of \textit{ additional} mutations, due to the assumption $\theta_N= o(r_N)$. Indeed, even in a population with all hosts being in state $1$ except one whose state is in $U^{\eta, N}$,
the population mutation rate to type $B$ is proportional to $ u_NNg_NM = \theta_N$, whereas the rate at which type $B$ is transmitted into the population proportional to $r_N$. Thus, due to
the assumption $\theta_N= o(r_N)$, even in this extreme scenario the effect of mutation is asymptotically negligible compared to that of reinfection.

(i) For proving  \eqref{MonoToPoly} we first consider a  population in which all hosts originally carry only type $A$-parasites. Immediately after in one single parasite (at time $0$, say) a mutation to type $B$ has occurred, the empirical distribution of host states is $\mu^N_0 := (1-\frac 1M)\delta_1 + \frac 1M \delta_{\frac 1N}$. (For the sake of readability we will sometimes write $M$ instead of $M_N$ for the number of hosts.) Let $p_N$ be the probability that, starting from this  ``nearly monomorphic'' $\mu^N_0$, the population turns 
into a ``nearly stable polymorphic'' one, in the sense that the triplet $(\mu^N_t(\{0\}), \mu^N_t(U^{\eta,N}), \mu^N_t(\{1\}))_{t\geq0}$ reaches the set $W^{\delta,\mathbf u}$ (defined in \eqref{Wdeltau}), and let $T$ be the random time which this transition takes. 
%\textcolor{green}{We first discuss the consequences of the assumption $\theta_N= o(r_N)$. In a population with all hosts being in state $0$ except one whose state is in $U^{\eta, N}$, the population mutation rate to type $A$ is $u_NNg_NM = \theta_N$, whereas the rate at which type $A$ is transmitted into the population is $\sim r_N\eta$. Thus, under
%the assumption $\theta_N= o(r_N)$, even in this extremal case mutation is dominated by reinfection, and thus the assertions of Theorem  \eqref{MF} remain valid. } 
% which $X^{N,M_N}$ takes to reach the set
%$W^{\delta,N}$ after a successful mutant occurred.
 
 To arrive at \eqref{MonoToPoly} we show that
\begin{itemize}
\item[a)]  $p_N$ is bounded from below by $c_1 s_N + o(s_N)$ for some constant $c_1= c_1(\eta, r)$ not depending on $N$,
\item[b)]  For any $\gamma >0$, $T \le (M_N)^{\gamma}$ whp.
\end{itemize}
From a) and b) together it follows that the waiting time to reach a polymorphic state 
from a monomorphic (pure type $A$) one can whp be estimated from above by the sum of $T$ and an exponentially distributed random variable
with parameter $\theta_N p_N$.  By analogy, the same reasoning applies to an initially pure type $B$ population, and altogether proves \eqref{MonoToPoly}. 

We now turn to the proof of claim a). 

 1. In the first step we consider the probability that the frequency
of type $A$  parasites in the host that was affected by the parasite mutation declines from $1-\frac 1N$ to  $\lfloor \eta + s_N^{a +\epsilon_1}\rfloor$.  By Lemma \ref{balprob} this probability is
$2 (1-\eta)s_N + o(s_N) $. Let us call such a mutation \textit{ locally successful}. 

2. In this second step we will investigate the probability that such a mutation is also \textit{ globally successful}, in the sense  that $(\mu^N_t(\{0\}), \mu^N_t(U^{\eta,N}), \mu^N_t(\{1\}))_{t\geq0}$ reaches the set~$W^{\delta,\mathbf u}$.   Specifically we will show that, starting from a $\mu_0^N$ with $\mu_0^N(U^{\eta,N})=\frac 1M$ and $\mu_0^N(\{1\})=1-\frac 1M$, this probability   can be estimated from below by 
a constant $\tilde c= \tilde c(\eta, r) >0$ for all sufficiently large $N\in \mathbbm{N}$.

To this purpose we first set out to estimate the probability to reach a (small) fraction $\delta_2$ of hosts with states in $\{0\}\cup U^{\eta,N}$. We will couple the number $M\cdot \mu_t^N(\{0\}\cup U^{\eta,N})$
with a supercritical branching process starting in a single ancestor.
To this end consider the following modification~$\tilde{X}^N$ of $X^{N,M_N}$:
\begin{itemize}
\item Ignore all ineffective reinfection events that affect hosts in state 1.
\item If a host in state $0$ appears (due to some host replacement event),  then switch its type instantly to a state in $D^{\eta,N}=[\eta-s_N^{a +\epsilon_1}, \eta+ s_N^{a+\epsilon_1}]$.
\item If a host reaches a state $[0,1)\backslash(\eta- s_N^a, \eta+s_N^a)$, then switch its state instantly to
state~1.
\end{itemize}

For each sequence $x_N \in D^{\eta,N}$, the rate at which the state of a host starting from $x_N$ leaves $U^{\eta,N}$ by random 
fluctuations, converges to 0 as $N\rightarrow \infty$  by Lemma
\ref{stab}. For each $N$ let $l_N$ 
be an upper bounded of this rate with  $l_N\to 0$ as $N\to \infty$.

Write $\tilde{Z}^{\eta,N}_t:= |\{i: \tilde X_i^{N} \in U^{\eta,N}, 1\le i\le M\}| $. The rate at which the process $\tilde{Z}^{\eta,N}$ jumps from~$k$ to 
$k+1$ is
\begin{equation}\label{up}2 r(1+ o(1)) (1-\eta)(1-\eta + \mathcal O(s_N)) \frac{k}{M} (M-k) + \frac{ k}{M} (1-\eta + \mathcal O(s_N)) (M-k) 
\end{equation}
and the rate at which the process $\tilde{Z}^{\eta,N}$ jumps from  $k$ to 
$k-1$ is
not larger than 
\begin{equation}\label{down}
\frac{k}{M}(M-k + (\eta + O(s_N)) k + l_N).
\end{equation}
To see \eqref{up}, note that the number of hosts with state in $U^{\eta,N}$ increases by 1  if a host with state in $U^{\eta,N}$ reinfects effectively
a host in state $1$ and transmits type $B$, or a host with state in $U^{\eta,N}$ replaces a host in state 0 and transmits type $B$ (since in this case
immediately the host state is changed from 0 to type $U^{\eta,N}$).
The asymptotic estimate \eqref{down} can be explained similarly.

When $X^{N,M_N}$ and $\tilde{X}^N$ start in the same configuration, then the process $\tilde{Z}^{\eta,N}$ is (asymptotically as $N\to \infty$) stochastically smaller than the process  $M_N \mu^N_t(\{0\}\cup U^{\eta,N})_{t\geq 0}$ as long as both processes are between $0$ and $M_N\delta'$ for $\delta'$ small enough  -- the reason being that in this case
the average increase of hosts in states 0  or $U^{\eta,N}$ is lowest if all host states are in $U^{\eta,N}$.  This asymptotic statement can be read off
from \eqref{dynsys}: if   $\delta' >0$ is small enough, then for all $0 \leq x < \delta'$ with $v^0+ v^\eta=x$
we have $\dot{v}^0(v^0, v^\eta, 1-x) + \dot{v}^\eta(v^0, v^\eta, 1-x) \geq \dot{v}^0(0, x, 1-x) + \dot{v}^\eta(0, x, 1-x)$.

 Now let $\zeta^N$  be a Markovian jump process on the natural numbers  which jumps from $k$ to  $k+1$ at rate $  k(2r(1-\eta)^2(1-\delta')+(1-\delta')(1-\eta)+o(1))$
 and from $k$ to $k-1$ at rate $k(1+\eta\delta' +o(1)).$
 
 The form of the jump rates of $\tilde{Z}^{\eta,N}$ allows to couple  $\tilde{Z}^{\eta,N}$ and $\zeta^N$ such that $\tilde{Z}^{\eta,N} \ge \zeta^N$ with high probability
 provided both processes have the same starting point, and as long as both processes are smaller than $M_N\delta' $.

From the just derived rates one checks  (using the inequality $r > \frac{\eta}{2(1-\eta)^2}$ which is part of the theorem's assumption \eqref{brachcoup}) that for sufficiently small $\delta'$ and $N$
large enough, 
the process $\zeta^N$ has a strictly positive linear drift on $1, \ldots, \lceil M_N\delta'\rceil$. Hence one
can couple $\zeta^N$
with a supercritical branching process $\zeta$, and one concludes that the probability that $\mu^N(\{0\}\cup U^{\eta,N})$  reaches the level~$\delta'$ can be estimated
by the survival probability $\tilde c(r,\eta)$ of the branching process $\zeta$.

3. Combining 1. and 2. we have proved that the probability to reach from a single parasite mutation (to type $B$) a  frequency $\delta'$
of hosts with states in  $\{0\}\cup U^{\eta, N}$ is not less than $(2 (1-\eta)s_N + o(s_N)) \tilde c(r,\eta) =: c_1 s_N+o(s_N)$. The fact that the latter is  also an asymptotic lower bound for $p_N$ (defined at the beginning of this proof) is now a direct consequence of Proposition~\ref{ergodic}. This finishes the proof of claim a).
\\\\
\indent
Next we turn to the proof of claim b).  To this purpose we estimate from below the time which the process $\mu^N_t(\{0\}\cup U^{\eta,N})_{t\geq0}$ needs to reach the level $\delta'$, by decomposing this time according to the above steps 1 and 2.

4. The time which the process $\mu^N_t(\{0\}\cup U^{\eta,N})_{t\geq 0}$ needs to reach the level $1/M$ corresponds to the time it takes for the frequency of the (locally) successful mutant's offspring  to rise from $1/N$ to $(1-\eta)-s_N^{a +\epsilon_1}$ (in the host that was
affected by the mutation). By Proposition \ref{TimeToEta} (which was already used for a corresponding argument in the proof of Theorem \ref{TfiniteM}) this time is asymptotically negligible on the host time scale as $N \to \infty$. 

5. In order to estimate the time which the process $\mu^N_t(\{0\}\cup U^{\eta,N})_{t\geq0}$  then needs to reach the level~$\delta'$, we
again use the supercritical branching process $\zeta$ from step 2. Let $\hat \zeta$ be that process conditioned to non-extinction. We thin $\hat \zeta$ by considering only its immortal lines; this amounts to decreasing the birth rate
by the positive factor $\tilde c(r,\eta)$, and renders a Yule process with a positive rate. Hence, 
the time until $\hat \zeta$ reaches the level $\delta'M_N$ can be estimated from above by $(M_N)^{\gamma}$
whp.  This completes also the proof of claim b), and, as we already stated before step 1, shows \eqref{MonoToPoly}.
\\\\
ii) We now turn to the proof of \eqref{partii}, i.e.~the second part of the theorem.
   For $\mu^N_t(\{1\})_{t\geq 0}$ to hit~1, when $(\mu^N_t(\{0\}), \mu^N_t(U^{\eta,N}), \mu^N_t(\{1\}))_{t\geq0}$ starts from an element of $W^{\delta,\mathbf u}$,  the process $M_N \mu^N_t(\{0\}\cup U^{\eta,N})_{t\geq 0}$ has to visit
some $\lceil \delta' M_N\rceil$, with a small enough $\delta'> 0$. Starting from then, we compare the process $M_N \mu^N_t(\{0\}\cup U^{\eta,N})_{t\geq 0}$ with
the above defined process $\zeta^N$, see part i) a) 2. of the proof. Because of the comparison arguments given there, it will be helpful to compute
\[p_{\delta'}:= \mathbb P_{\lfloor M_N\delta'/2 \rfloor }\big (\zeta^N \mbox{ hits } \lceil M_N\delta'/4 \rceil \mbox{ before it hits  } \lfloor M_N3\delta'/4 \rfloor\big ).\]
An inspection of the jump rates of $\zeta^N$ given in step 2 of part i) of the proof (and using the assumption $r> \frac{\eta}{2(1-\eta)^2}$ which is part of the theorem's assumption \eqref{brachcoup}) shows that for $\delta'$ small enough there exists a $c_1 >0$ such that $\zeta^N$ has between $\lceil M_N\delta'/4 \rceil$ and
$\lfloor M_N3\delta'/4 \rfloor$ an
upward drift $\ge c_1$. Lemma \ref{ruin} therefore gives the existence of a $c >1$ such that 
\[ p_{\delta'} \leq  \frac{c^{\lceil M_N \delta'/4 \rceil } -1}{c^{\lfloor M_N \delta'/2 \rfloor} -1} \sim \exp(-(\log c) M_N\delta'/4).\]

Hence the time $\tau^N$ which $\mu^N_t(\{1\})_{t\geq 0}$ needs to hit 1, when  $(\mu^N_t(\{0\}), \mu^N_t(U^{\eta,N}), \mu^N_t(\{1\}))_{t\geq0}$ is initially in $W^{\delta,\mathbf u}$, can whp be estimated from below by $\sum_{j=1}^G H_j$ where $G$ is a
geometrically distributed variable with success probability $\exp(-(\log c) M_N \delta'/4)$ and $H_j$ are independent copies of the time which $\zeta^N$ needs to
reach $\{\lceil M_N \delta'/4 \rceil, \lfloor M_N 3\delta'/4 \rfloor\}$ when starting in $\lfloor M_N \delta'/2 \rfloor$. From this it  follows that whp 
\[\tau^N > \exp( (M_N)^{1-\gamma}).\]

Analogously, one arrives at an estimate for the time which $\mu^N_t(\{0\})_{t\geq0}$ needs to hit 1. Hence the claim follows.

\hfill $\Box$

\subsection*{Acknowledgements}
This research was supported by the DFG priority programme 1590, in particular through grant WA-947/4-2 to AW. We thank
Josef Hofbauer for help concerning the stability of the dynamical system \eqref{dynsys}. We thank two anonymous referees for valuable comments on the manuscript.

%\section*{References}
\bibliography{mybibfile}

\end{document}